\documentclass{amsart}
\usepackage{amstext}
\usepackage{amsthm}
\usepackage{amsmath}
\RequirePackage{amsfonts}
\RequirePackage{amssymb}
\usepackage{graphicx}
\RequirePackage{xcolor}
\RequirePackage{enumerate} 
\usepackage{tikz}
\usepackage{subcaption}
\usepackage{caption}

\newtheorem{theorem}{Theorem}[section]
\newtheorem{lemma}[theorem]{Lemma}

\newtheorem{corollary}[theorem]{Corollary}
\theoremstyle{definition}\newtheorem{remark}[theorem]{Remark}
\theoremstyle{definition}\newtheorem{example}[theorem]{Example}

\allowdisplaybreaks

\newcommand{\D}{\mathbb{D}}
\newcommand{\C}{\mathbb{C}}
\renewcommand{\H}{\mathbb{H}}
\newcommand{\R}{\mathbb{R}}

\newcommand{\abs}[1]{\left| #1 \right|}

\numberwithin{equation}{section}

\title[Rates of convergence]{On the rates of convergence of orbits in semigroups of holomorphic functions}
\author[D. Betsakos]{Dimitrios Betsakos}
\address{Department of Mathematics, Aristotle University of Thessaloniki, 54124, Thessaloniki, Greece}
\email{betsakos@math.auth.gr}
\author[F. J. Cruz-Zamorano]{Francisco J. Cruz-Zamorano}
\address{Departamento de Matem\'{a}tica Aplicada II and IMUS, Escuela T\'{e}cnica Superior de Ingenier\'{i}a, Universidad de Sevilla, Camino de los Descubrimientos, S/N, 41092, Seville, Spain}
\email{fcruz4@us.es}
\author[K. Zarvalis]{Konstantinos Zarvalis}
\address{Instituto de Matem\'{a}ticas de la Universidad de Sevilla, Avenida de Reina Mercedes, S/N, 41012, Seville, Spain}
\email{kzarvalis@us.es}

\date{}
\thanks{F. J. Cruz-Zamorano was supported by Ministerio de Innovaci\'on y Ciencia, Spain, project PID2022-136320NB-I00 and by Ministerio de Universidades, Spain, through the action Ayuda del Programa de Formaci\'on de Profesorado Universitario, reference FPU21/00258. This work was initiated when F. J. Cruz-Zamorano visited the Aristotle University of Thessaloniki, funded by Ministerio de Ciencia, Innovaci\'on y Universidades, reference EST24/00450.}
\thanks{K. Zarvalis was supported by Junta de Andaluc\'{i}a, grant number QUAL21 005 USE.}

\begin{document}

\begin{abstract}
    Let $(\phi_t)$ be a continuous semigroup of holomorphic self-maps of the unit disk $\D$ with Denjoy-Wolff point $\tau\in\overline{\D}$. We study the rate of convergence of the forward orbits of $(\phi_t)$ to the Denjoy-Wolff point by finding explicit bounds for the quantity $|\phi_t(z)-\tau|$, $z\in\overline{\D}$, $t>0$. We further discuss the corresponding rate of convergence for the backward orbits of $(\phi_t)$.
\end{abstract}

\maketitle
\tableofcontents

\section{Introduction}
One-parameter continuous semigroups of holomorphic self-maps of the unit disk have been extensively studied for many years. Even though the principal results of its theory were established quite early, the field has met a considerable flourishing in the past two decades.

A one-parameter continuous semigroup of holomorphic self-maps of the unit disk $\D$ (or from now on, a \textit{semigroup in} $\D$) is a family $(\phi_t)$ of holomorphic functions $\phi_t:\D\to\D$, $t\ge0$, which satisfy the following three conditions:
\begin{enumerate}
\item[(i)] $\phi_0=\textup{id}_\D$;
\item[(ii)] $\phi_{t+s}=\phi_t\circ\phi_s$, for all $t,s\ge0$;
\item[(iii)] $\lim\limits_{t\to 0^+}\phi_t(z)=z$, for all $z\in\D$.
\end{enumerate}
Semigroups are of great interest in their own right. However, they possess strong ties with a multitude of cutting edge objects in the theory of dynamical systems. For comprehensive presentations of the rich theory of semigroups, we refer to the books \cite{Abate,BCDM,ES} and references therein.
 
Fix $z\in\D$. The function $\gamma_z:[0,+\infty)\to\D$ with $\gamma_z(t)=\phi_t(z)$ is called the \textit{forward orbit} of $z$, which is its \textit{starting point}. During the course of the present work, we will interchangeably use the term forward orbit for both the function $\gamma_z$ and the set $\{\phi_t(z):t\ge0\}$ without any confusion.
 
We note here that in case there exists some $t_0>0$ such that the function $\phi_{t_0}$ is a conformal automorphism of the unit disk, then the same is necessarily true for every member of the semigroup and we say that $(\phi_t)$ is actually a \textit{group}; see \cite[Theorem 8.2.4]{BCDM}.
 
Naturally, a focal point in the research is the asymptotic behavior of the orbits as $t\to+\infty$. A stunning result in the dynamics of semigroups is the continuous version of the Denjoy-Wolff Theorem; see \cite[Theorem 5.5.1]{Abate} or \cite[Theorem 8.3.1]{BCDM}. According to this, for every semigroup $(\phi_t)$ which is not a group,
there exists a unique point $\tau\in\overline{\D}$ such that
\begin{equation}\label{eq:DW}
\lim\limits_{t\to+\infty}\gamma_z(t)=\lim\limits_{t\to+\infty}\phi_t(z)=\tau, \quad\text{for all }z\in\D.
\end{equation}
This point $\tau$ is called the \textit{Denjoy-Wolff point} of $(\phi_t)$ and is inextricably linked with the study of the semigroup. In fact, its position on the closure of the unit disk provides a first distinction within the class of semigroups:
\begin{enumerate}
\item[(i)] if $\tau\in\D$, $(\phi_t)$ is called \textit{elliptic},
\item[(ii)] if $\tau\in\partial\D$, $(\phi_t)$ is called \textit{non-elliptic}.
\end{enumerate}
The Denjoy-Wolff Theorem remains valid for non-elliptic groups as well, but not for elliptic groups. 

Non-elliptic semigroups may be further distinguished in several types depending on angular derivatives and hyperbolic geometry. More information about the distinct types of semigroups follows in the next section.

In addition to the forward dynamics, in recent years there has been a considerable advancement in the theory of \textit{backward} dynamics of semigroups, see e.g. \cite{BCDMG, ESZ, GKR, KZ} and \cite[Chapter 13]{BCDM}. Given $z\in\D$, it is always possible to define its \textit{backward orbit} through the following method: the function $\phi_t^{-1}$ is well-defined
at $z$ for every $t$ in an interval of the form $[0,T)$. Let $T_z$ be the supremum of all such numbers $T$. Then, the continuous curve $\widetilde{\gamma}_z:[0,T_z)\to\D$ with
\begin{equation}\label{eq:backward definition}
\widetilde{\gamma}_z(t)=\phi_t^{-1}(z)
\end{equation}
is called the backward orbit of $(\phi_t)$ with \textit{starting point} $z$ (cf. \cite[Definition 13.1.1]{BCDM}). For the purposes of this work and as usual in the literature, we will only deal with backward orbits emanating from points $z\in\D$ satisfying $T_z=+\infty$.

Backward orbits are more volatile than their forward counterparts, presenting more intricacies. The most striking difference is that not all backward orbits converge to the same point. However, we know (cf. \cite[Lemma 13.1.5]{BCDM}) that each non-trivial backward orbit $\widetilde{\gamma}_z$ with $T_z=+\infty$ converges, as $t\to+\infty$, to some point $\sigma\in\partial\D$ (it can even be $\sigma=\tau$) regardless of the type of the semigroup (even if the semigroup is elliptic).

One of the fundamental fields of research in semigroups is the so-called \textit{rate of convergence} of the forward orbits to the Denjoy-Wolff point of the semigroup. We already know that given a semigroup $(\phi_t)$ in $\D$, all its forward orbits converge to $\tau$ because of (\ref{eq:DW}). It is only logical to wonder about how fast this convergence happens. In order to study this problem, we inspect the quantity $|\gamma_z(t)-\tau|=|\phi_t(z)-\tau|$, $z\in\D$, $t>0$, trying to find bounds depending on $t$. In the past few years, there have been a lot of contributions concerning this subject; see \cite{DB-Rate-H, DB-Rate-P, BCDM-Rate, BK, BCK, JLR, KTZ}. See also \cite[Chapter 16]{BCDM}, \cite[Chapter 7]{ES} and references therein. In most of those works, the authors find how these bounds depend on the type of the semigroup, while also demonstrating the sharpness of their results.

A natural next step is to inspect the rates of convergence for backward orbits, as well. Some partial results about this subject appear in  \cite{BCDMG, GKR}. The focal point of the current work is to perform a more thorough study of this topic. Since not all backward orbits converge to the same point, this time we examine the quantity $|\widetilde{\gamma}_z-\sigma|$, where $\sigma=\lim_{t\to+\infty}\widetilde{\gamma}_z(t)$ (recall that we only deal with $z\in\D$ for which $T_z=+\infty$). Our aim is to see how the respective rate is influenced by the various attributes of the semigroup and the backward orbit itself. In addition, one goal is to render explicit the dependence of the bounds on the starting point in order to understand how the position of the starting point impacts the rate. 

In conclusion, the main objective of the present article is to find lower and upper bounds based on $t>0$ for the quantity $|\widetilde{\gamma}_z-\sigma|$, where $\sigma\in\partial\D$ is the point that the backward orbit $\widetilde{\gamma}_z$ lands at. In the meantime, the techniques used in backward orbits allow us to provide certain improvements concerning the already known bounds for forward orbits. All in all, we try to provide a complete picture of both upper and lower bounds for all kinds of orbits and semigroups, all the while finding explicitly the constants that govern these rates. Whenever applicable, we will present the respective sharpness examples.

Because of the vast difference in the techniques, we will treat elliptic and non-elliptic semigroups separately. We start with non-elliptic semigroups and we mainly use techniques based on harmonic measure and hyperbolic geometry. Then, we prove analogue results for elliptic semigroups using arguments that are, in a sense, more straightforward.

\section{Statement of Main Results}
In this section, we are going to state all our main results providing beforehand all the necessary notions required for the comprehension of the statements.

We start with a very important tool that possesses a prominent role in the study of semigroups in $\D$. For a non-elliptic semigroup $(\phi_t)$ in $\D$,  there exists a unique (up to translation) conformal mapping $h:\D\to\C$ such that
\begin{equation}\label{eq:koenigs}
h(\phi_t(z))=h(z)+t, \quad\text{for all }z\in\D\text{ and all }t\ge0.
\end{equation}
This mapping $h$ is called the \textit{Koenigs function} of the semigroup, while the simply connected domain $\Omega:=h(\D)$ is its \textit{Koenigs domain}. The Koenigs function $h$ of a non-elliptic semigroup $(\phi_t)$ greatly facilitates its study since it linearizes the trajectories rendering their examination simpler. An elliptic semigroup also has a Koenigs map which is defined in a similar manner, but for the purposes of this article we will not need it. For more information on the Koenigs function, we refer the interested reader to \cite[Chapter 9]{BCDM}.

We proceed to the classification of semigroups. Given an elliptic semigroup $(\phi_t)$ with Denjoy-Wolff point $\tau\in\D$, by \cite[Theorem 8.3.1]{BCDM} we know that there exists some $\lambda\in\C$ with $\textup{Re}\lambda>0$ such that $\phi_t'(\tau)=e^{-\lambda t}$, for all $t\ge0$. In the non-elliptic case, an analogue statement is true but the position of the Denjoy-Wolff point requires the use of angular derivatives. Indeed, given a non-elliptic semigroup $(\phi_t)$ with Denjoy-Wolff point $\tau\in\partial\D$, there exists a non-negative real number $\lambda$ such that $\angle\lim_{z\to\tau}\phi_t'(z)=e^{-\lambda t}$, for all $t\ge0$. In particular, $(\phi_t)$ is said to be \textit{hyperbolic} whenever $\lambda>0$ and \textit{parabolic} whenever $\lambda=0$. In all types of semigroups, this unique number $\lambda$ is called the \textit{spectral value} of the semigroup.

An essential classification of non-elliptic semigroups can de derived through the Koenigs domain $\Omega$ looking at whether it is included in certain standard horizontal domains (see Theorem \cite[Theorem 9.3.5]{BCDM}).
\begin{enumerate}
    \item[(i)] $(\phi_t)$ is hyperbolic if $\Omega$ is contained in a horizontal strip,
    \item[(ii)] $(\phi_t)$ is called \textit{parabolic of positive hyperbolic step} if $\Omega$ is contained in a horizontal half-plane, but not in a horizontal strip,
    \item[(iii)] and $(\phi_t)$ is called \textit{parabolic of zero hyperbolic step} if $\Omega$ is not contained in any horizontal half-plane.
\end{enumerate}
In addition, if $(\phi_t)$ is hyperbolic with spectral value $\lambda>0$, then the smallest horizontal strip containing $\Omega$ is of width $\pi/\lambda$. The usual definition of hyperbolic step is different, but for our purposes, this counterpart is adequate.

It is frequent in the study of non-elliptic semigroups to proceed to the following normalizations for a semigroup $(\phi_t)$ with Koenigs function $h$ and Koenigs domain $\Omega$:
\begin{enumerate}
    \item[(i)] whenever $(\phi_t)$ is hyperbolic with spectral value $\lambda$, we assume that $\textup{Re}h(0)=0$ and that the smallest horizontal strip containing $\Omega$ is $S=\{w\in\C:0<\textup{Im}w<\pi/\lambda\}$,
    \item[(ii)] whenever $(\phi_t)$ is parabolic of positive hyperbolic step, we assume that $\textup{Re}h(0)=0$ and that the smallest horizontal half-plane containing $\Omega$ is either the upper half-plane $\H$ or the lower half-plane $-\H$,
    \item[(iii)] and whenever $(\phi_t)$ is parabolic of zero hyperbolic step, we assume $h(0)=0$.
\end{enumerate}

Summing up, the four main types of semigroups in $\D$ are elliptic, hyperbolic, parabolic of positive hyperbolic step and parabolic of zero hyperbolic step. There exist further classifications within the classes of hyperbolic semigroups and parabolic semigroups of positive hyperbolic step which we will mention in subsequent sections. Nevertheless, during the course of this article, the four main types will mostly suffice.

In the same vein, the main distinction within the class of backward orbits comes again through notions of hyperbolic geometry. More specifically, a backward orbit $\widetilde{\gamma}_z$ with $T_z=+\infty$ is said to be \textit{regular} provided 
$$\limsup\limits_{t\to+\infty}d_\D(\widetilde{\gamma}_z(t),\widetilde{\gamma}_z(t+1))<+\infty.$$
Whenever the above upper limit is infinite, we say that the backward orbit is \textit{non-regular}.

Consider the set of all points $z\in\D$ satisfying $T_z=+\infty$. A non-empty connected component of the interior of this set is called a \textit{petal} of $(\phi_t)$ (cf. \cite[Definitions 13.3.4 and 13.4.1]{BCDM}). All the backward orbits contained inside a petal are regular and converge to the same point $\sigma$ of the unit circle; see \cite[Proposition 13.4.2]{BCDM}. If $\sigma=\tau$, then the petal is said to be \textit{parabolic}. Otherwise, we call it \textit{hyperbolic}. It is worth mentioning that parabolic petals can only appear in parabolic semigroups. 

Our first main result deals with regular backward orbits of non-elliptic semigroups. We find their rates of convergence to their endpoints and examine how the type of the semigroup and the type of the petal influence the rate. In addition, we find explicitly the constants that govern this convergence in order to understand how the starting point impacts the rate. We will prove the following:

\begin{theorem}[\textbf{Non-elliptic semigroups, regular backward orbits}]
\label{thm:backward}
Let $(\phi_t)$ be a non-elliptic semigroup in $\D$ with Denjoy-Wolff point $\tau$ and Koenigs function $h$. Suppose that $\Delta$ is a petal of $(\phi_t)$ and that all backward orbits contained in $\Delta$ converge to $\sigma\in\partial\D$. Then, for every $z \in \Delta$ and all $t > 1$, the following hold:
\begin{enumerate}[\hspace{0.5cm}\normalfont(a)]
\item If $\Delta$ is parabolic (and hence $\sigma=\tau$) and $(\phi_t)$ is parabolic of zero hyperbolic step, then there exists $\alpha\in\R$ depending on $\Delta$ such that
$$\frac{|z-\tau|\min\{1,|\textup{Im}h(z)-\alpha|\}}{2(1+|z|)}\frac{1}{t} \leq \abs{\widetilde{\gamma}_z(t)-\tau} \leq 4\sqrt{2}\pi\left(\abs{h'(0)}^{1/2}+\abs{h(z)}^{1/2}\right)\dfrac{1}{\sqrt{t}}.$$
\item If $\Delta$ is parabolic (and hence $\sigma=\tau$) and $(\phi_t)$ is parabolic of positive hyperbolic step, then there exists $\alpha\in\R$ depending on $\Delta$ such that
$$\frac{|z-\tau|\min\{1,|\textup{Im}h(z)-\alpha|\}}{2(1+|z|)}\frac{1}{t} \leq\abs{\widetilde{\gamma}_z(t)-\tau} \leq 8\max\{\abs{h(0)},\abs{h(z)}\}\dfrac{1}{t}.$$
\item If $\Delta$ is hyperbolic (and hence $\sigma\ne\tau$), then there exist $\nu \in (-\infty,0)$ depending on $\sigma$ and $\alpha\in\R$ depending on $\Delta$ such that for every $\epsilon \in (0,-\nu)$ there exists $T=T(z,\epsilon)>0$ (which can be explicitly stated) such that
$$\frac{1-|z|}{1+|z|}\sin^2\left[\nu(\textup{Im}h(z)-\alpha)\right]e^{\nu t} \leq \abs{\widetilde{\gamma}_z(t)-\sigma} \leq 16e^{-(\nu+\epsilon)T}e^{(\nu+\epsilon)t}.$$
\end{enumerate}
\end{theorem}
The real number $\alpha$ and the positive number $T$ depend on the geometry of the Koenigs domain of $(\phi_t)$ and of the image $h(\Delta)$. Their actual meaning will be explained during the proof. In addition, the number $\nu$ concerns an important notion in backward dynamics. More information on backward orbits and the exact definitions of petals follow in Section 2. Furthermore, through a series of examples, we will demonstrate the sharpness of all the above bounds.

To the authors' best knowledge not many results with regard to the rates of convergence of backward orbits have appeared in the literature. Some partial results were analyzed in \cite[Proposition 4.20]{BCDMG}, and a recent discussion appeared in \cite[Subsection 6.1]{GKR}. However, these works only deal with regular orbits on hyperbolic petals. The main advantange of our work is that we can push our techniques further, proving similar results even for non-regular orbits. As a matter of fact:

\begin{theorem}[\textbf{Non-elliptic semigroups, non-regular backward orbits}]\label{thm:non-regular}
    Let $(\phi_t)$ be a non-elliptic semigroup in $\D$ with Denjoy-Wolff point $\tau$ and Koenigs function $h$. Let $z\in\D$ and suppose that $\widetilde{\gamma}_z:[0,+\infty)\to\D$ is a non-regular backward orbit for $(\phi_t)$. Then, for all $t>1$, the following hold:
    \begin{enumerate}
        \item[\textup{(a)}] If $\widetilde{\gamma}_z([0,+\infty))$ lies on the boundary of a parabolic petal (and hence $\widetilde{\gamma}_z$ converges to $\tau$) and $(\phi_t)$ is parabolic of zero hyperbolic step, then
        $$|\widetilde{\gamma}_z(t)-\tau|\le4\sqrt{2}\pi\left(|h'(0)|^{1/2}+|h(z)|^{1/2}\right)\frac{1}{\sqrt{t}}.$$
        \item[\textup{(b)}] If $\widetilde{\gamma}_z([0,+\infty))$ lies on the boundary of a parabolic petal (and hence $\widetilde{\gamma}_z$ converges to $\tau$) and $(\phi_t)$ is parabolic of positive hyperbolic step, then
        $$|\widetilde{\gamma}_z(t)-\tau|\le8\max\{|h(0)|,|h(z)|\}\frac{1}{t}.$$
        \item[\textup{(c)}] If $\widetilde{\gamma}_z([0,+\infty))$ lies on the boundary of a hyperbolic petal and $\widetilde{\gamma}_z$ converges to $\sigma$ (and hence $\sigma\in\partial\D\setminus\{\tau\}$), then there exists $\nu\in(-\infty,0)$ depending on $\sigma$ such that for every $\epsilon>0$ there exists $C=C(z,\epsilon)>0$ (which can be explicitly stated) such that
        $$|\widetilde{\gamma}_z(t)-\sigma|\le C e^{(\nu+\epsilon)t}.$$
        \item[\textup{(d)}] If $\widetilde{\gamma}_z([0,+\infty))$ does not lie on the boundary of any petal and $\widetilde{\gamma}_z$ converges to $\sigma$ (and hence $\sigma\in\partial\D\setminus\{\tau\}$), then for every $\epsilon > 0$ there exists $C=C(z,\epsilon) > 0$ (which can be explicitly stated) such that
        $$|\widetilde{\gamma}_z(t)-\sigma|\le Ce^{-\epsilon t}.$$
    \end{enumerate}
\end{theorem}
We avoid writing the explicit constants in the last two cases for the sake of legibility. They will be stated inside the proofs for better comprehension. More information on non-regular backward orbits and the reason why we proceeded to the above distinction follow in subsequent sections. Furthermore, the reader may notice that we did not provide lower bounds. Through Example \ref{ex:non-regular} we will verify that non-regular backward orbits of non-elliptic semigroups may converge arbitrarily fast to the point at which they are going to land.

Other than producing the rates for the backward orbits of non-elliptic semigroups, the techniques used in the proofs of the previous results, allow us to improve the already known results for forward orbits. The dependence of the rates of forward orbits with respect to $t$ and the corresponding sharpness appear in the literature (see for instance \cite{DB-Rate-H,DB-Rate-P,BCDM-Rate}). However, our goal is to find explicit constants that govern the convergence to the Denjoy-Wolff point in order to better comprehend exactly what parameters influence the rate of the convergence. 

Before stating our next result, we need to mention some results concerning the boundary behavior of semigroups. A theorem proved by Gumenyuk in \cite[Theorem 3.1]{Gumenyuk} yields that for every $t\ge0$, the function $\phi_t$ has finite non-tangential limit at {\it every} point of the unit circle and that for every $\zeta\in\partial\D$ the function
$$[0,+\infty)\ni t\mapsto\phi_t(\zeta):=\angle\lim\limits_{z\to\zeta}\phi_t(z)$$
is still continuous. Therefore, it makes sense to enquire about the forward orbit $\gamma_z$ for every point $z\in\overline{\D}$. On the contrary, the backward orbit $\widetilde{\gamma}_z$ is not always well-defined for $z\in\partial\D$.

In connection with the result of Gumenyuk, we should point out that the angular limit $\angle\lim_{z \to \zeta}h(z)$ exists in $\C \cup \{\infty\}$ for all $\zeta \in \partial\D$; see \cite[Corollary 11.1.7]{BCDM}. Moreover, by \cite[Proposition 13.6.1]{BCDM}, we know that $\angle\lim_{z \to \zeta}h(z) = \infty$ if and only if $\zeta \in \partial\D$ is a \textit{boundary fixed point} of $(\phi_t)$, that is, $\angle\lim_{z \to \zeta}\phi_t(z) = \zeta$ for all $t \geq 0$. For non-elliptic semigroups, one such boundary fixed point is always the Denjoy-Wolff point of the semigroup. Therefore, due to its convenience, we define
$$h(\zeta) := \angle\lim_{z \to \zeta}h(z) \in \C$$
whenever $\zeta \in \partial\D$ is not a boundary fixed point of $(\phi_t)$ and $h(\zeta):=\infty$ whenever $\zeta\in\partial\D$ is a boundary fixed point. In particular, if $(\phi_t)$ is non-elliptic and $\zeta$ is a boundary fixed point different than the Denjoy-Wolff point, then we may also write $\textup{Re}h(\zeta):=-\infty$ (cf. \cite[Proposition 13.6.2]{BCDM}).

Now we are ready to proceed to our theorems about forward orbits of non-elliptic semigroups.
\begin{theorem}[\textbf{Non-elliptic semigroups, forward orbits, upper bounds}]\label{thm:forward}
Let $(\phi_t)$ be a non-elliptic semigroup in $\D$ with Denjoy-Wolff point $\tau$ and Koenigs function $h$. Then, for every $z \in \overline{\D}$ and all $t > 0$, the following hold:
\begin{enumerate}[\hspace{0.5cm}\normalfont(a)]
\item If $(\phi_t)$ is parabolic of zero parabolic step, then
$$\abs{\gamma_z(t)-\tau} \leq 4\sqrt{2}\pi\left(\abs{h'(0)}^{1/2}+\abs{h(z)}^{1/2}\right)\dfrac{1}{\sqrt{t}}.$$

\item If $(\phi_t)$ is parabolic of positive hyperbolic step, then
$$\abs{\gamma_z(t)-\tau} \leq 8 \max\{\abs{h(0)},\abs{h(z)}\}\dfrac{1}{t}.$$
\item If $(\phi_t)$ is hyperbolic with spectral value $\lambda \in (0,+\infty)$, then
$$\abs{\gamma_z(t)-\tau} \leq 16e^{-\lambda\mathrm{Re}h(z)}e^{-\lambda t}.$$
\end{enumerate}
\end{theorem}
Note here that in case $z$ is a boundary fixed point, then the result is trivial. Indeed, if $z=\tau$, then $|\gamma_\tau(t)-\tau|=|\tau-\tau|=0$, for all $t>0$, and the result holds at once. On the other hand, if $z=\sigma\ne\tau$ is a boundary fixed point, then $|h(z)|=|h(\sigma)|=+\infty$ and $\textup{Re}h(z)=\textup{Re}h(\sigma)=-\infty$, while the quantity $|\gamma_z(t)-\tau|$ is always equal to $|\sigma-\tau|\le 2$. Hence the result holds trivially as well.

At this point, we remind that upper bounds are already known; see for instance \cite{BCDM-Rate}. Our contribution lies in providing constants that render the dependence of the upper bounds on the starting point explicit. In our next theorem, we find sharp lower bounds for the quantity $|\gamma_z(t)-\tau|$ for non-elliptic semigroups. This time, the bounds can be presented in a unified way.

\begin{theorem}[\textbf{Non-elliptic semigroups, forward orbits, lower bound}]\label{thm:lower}
    Let $(\phi_t)$ be a non-elliptic semigroup in $\D$ with Denjoy-Wolff point $\tau$ and spectral value $\lambda\ge0$. Let $z\in\D$  and $\epsilon>0$. Then there exists a constant $c=c(z,\epsilon)>0$ (which can be explicitly stated) such that
    $$|\gamma_z(t)-\tau|\ge ce^{-(\lambda+\epsilon)t}, \quad\text{for all }t\ge0.$$
\end{theorem}
The constant $c$ depends on the geometry of the Koenigs domain of $(\phi_t)$. We avoid its explicit mention at this point in order to explain it more effectively during the respective proof. As we will see in Example \ref{ex:forward-orbits-lower-bound-sharpness}, the bound above is sharp for parabolic semigroups. For hyperbolic semigroups, the rate with respect to $t$ and its sharpness are known (see \cite{DB-Rate-H}).

Finally, we will work with the rate of convergence of forward and backward orbits for elliptic semigroups, as well. Since the Denjoy-Wolff points of elliptic semigroups lie inside the unit disk, we will only work with starting points $z\in\D$, because any forward orbit $\gamma_z$ will eventually enter $\D$ in order to reach $\tau$, while backward orbits are not necessarily well-defined for starting points on the unit circle. 
\begin{theorem}[\textbf{Elliptic semigroups, all orbits}]\label{thm:elliptic}
    Let $(\phi_t)$ be an elliptic semigroup in $\D$, not a group, with Denjoy-Wolff point $\tau$ and spectral value $\lambda\in\C$, $\textup{Re}\lambda>0$. Then:
\begin{enumerate}[\hspace{0.5cm}\normalfont(a)]
    \item For every $z\in\D$ and all $t\ge0$
    $$\dfrac{(1-|\tau|)|\tau-z|}{|1-\bar{\tau}z|}e^{-\mathrm{Re}\lambda \exp(2d_\D(\tau,z))t} \leq \abs{\gamma_z(t)-\tau} \leq \dfrac{(1+|\tau|)|\tau-z|}{|1-\bar{\tau}z|}e^{-\mathrm{Re}\lambda \exp(-2d_\D(\tau,z))t}.$$
    \item Suppose that $\Delta$ is a petal of $(\phi_t)$ and that all backward orbits contained in $\Delta$ converge to $\sigma\in\partial\D$. There exists $\nu\in(-\infty,0)$ depending on $\sigma$ such that for all $\epsilon>0$, $t\ge0$ and $z\in\Delta$
        $$c_1e^{\nu t} \leq |\widetilde{\gamma}_z(t)-\sigma| \leq c_2e^{(\nu+\epsilon)t},$$
    where $c_1=c_1(z)$ and $c_2=c_2(z,\epsilon)$ can be explicitly stated.
\end{enumerate}
\end{theorem}

The constants $c_1, c_2$ are found through the constants in Theorem \ref{thm:backward}(c). In the case when $\tau=0$, the exact same rate of convergence to the Denjoy-Wolff point can be found in \cite{Shoikhet}. For $\tau\in\D\setminus\{0\}$, the rate can be extracted through the use of an automorphism of $\D$. Despite the simplicity of the proof, we include the result for the sake of completeness. Furthermore, we will discuss the sharpness of the above rates in subsequent sections. Of course, elliptic semigroups can have non-regular backward orbits too. Similar statements to those in Theorem \ref{thm:non-regular}(c),(d) hold for the elliptic case as well (note that by \cite[Remark 13.4.5]{BCDM} elliptic semigroups cannot have parabolic petals). The proof for this may be derived identically to Theorem \ref{thm:elliptic}(b). For the sake of not being superfluous, we omit this case.

The structure of the article is as follows: First of all, in Section 3, we are going to mention some more general information with respect to  semigroups and the tools we will use to prove our results. Then, in Section 4, we discuss forward orbits of non-elliptic semigroups and their rates of convergence by proving Theorems \ref{thm:forward} and \ref{thm:lower}. Using the techniques developed in that section, we analyze the rates of convergence of backward orbits by proving Theorems \ref{thm:backward} and \ref{thm:non-regular} in Section 5. Finally, in Section 6 we work with elliptic semigroups and their respective orbits proving Theorem \ref{thm:elliptic}. In each section, we will lay out remarks and examples demonstrating the corresponding sharpness of the rates.

\section{Auxiliary Information}
\subsection{Semigroups in the Unit Disk}
We start the section with some additional information concerning the theory of semigroups. Let $(\phi_t)$ be a semigroup in $\D$. Any $z\in\D$ satisfying $\phi_t(z)=z$, for all $t\ge0$, is called a \textit{fixed point} of the semigroup. It is easy to check that only elliptic semigroups can have such a fixed point. In particular, an elliptic $(\phi_t)$ always has exactly one fixed point in $\D$, its Denjoy-Wolff point. On top of that, we are interested in fixed points lying on the unit circle. A $\zeta\in\partial\D$ is called a \textit{boundary fixed point} of $(\phi_t)$ if
\begin{equation*}
    \angle\lim\limits_{z\to\zeta}\phi_t(z)=\zeta, \quad\text{for all }t\ge0.
\end{equation*}
Both elliptic and non-elliptic semigroups can have boundary fixed points. Of course, one of the boundary fixed points of a non-elliptic semigroup is its Denjoy-Wolff point. For any other boundary fixed point $\sigma\in\partial\D$ of a semigroup $(\phi_t)$, we know that (see \cite[Proposition 12.1.6]{BCDM}) either there exists some $\nu\in(-\infty,0)$ such that $\angle\lim_{z\to\sigma}\phi'_t(z)=e^{-\nu t}$, for all $t\ge0$, or $\angle\lim_{z\to\sigma}\phi'_t(z)=\infty$, for all $t\ge0$. In case the former holds, we say that $\sigma$ is a \textit{repelling} fixed point. When the latter holds, $\sigma$ is said to be \textit{super-repelling}. For a repelling fixed point, the number $\nu$ provided above is called its \textit{repelling spectral value}. For super-repelling fixed points, we may say that their corresponding spectral value is $-\infty$. 

We move on from fixed points to a tool we already mentioned in Section 2, the Koenigs function. When $(\phi_t)$ is non-elliptic with Koenigs function $h$, we already wrote that $h(\phi_t(z))=h(z)+t$, for all $z\in\D$ and $t\ge0$. In particular, the Koenigs domain $\Omega:=h(\D)$ is \textit{convex in the positive direction}, i.e. $\Omega+t\subseteq\Omega$, for all $t\ge0$. The importance of this fact lies in its invertibility. More specifically, if $\Omega$ is any convex in the positive direction simply connected domain and $h:\D\to\Omega$ is a Riemann mapping, then the relation $\phi_t(z):=h^{-1}(h(z)+t)$ gives birth to a non-elliptic semigroup. For more information on the Koenigs function and its mapping properties, we refer to \cite[Chapter 9]{BCDM}. 

\subsection{Backward Orbits}
We move on to the backward dynamics of continuous semigroups in $\D$. In Section 2, we have already provided the definition of backward orbits and their regularity. In this subsection, we mostly follow \cite[Chapter 13]{BCDM}.

In the present article, we will only work with backward orbits emanating from points $z\in\D$ satisfying $T_z=+\infty$. So from now on, whenever we talk about the backward orbit $\widetilde{\gamma}_z$, we will always suppose a priori that $T_z=+\infty$, without the need of explicitly mentioning it.

We start with regular backward orbits. When we have an elliptic semigroup $(\phi_t)$ with Denjoy-Wolff point $\tau\in\D$, then the backward orbit $\widetilde{\gamma}_\tau$ satisfies $\widetilde{\gamma}_\tau(t)=\tau$, for all $t\ge0$. For the rest of the article, we disregard this trivial instance. In any other case, by \cite[Proposition 13.1.7]{BCDM}, we know that every backward orbit converges either to a repelling fixed point of the semigroup or to the Denjoy-Wolff point (this can only happen for non-elliptic semigroups).

Given a semigroup $(\phi_t)$ in $\D$, the set $W:=\cap_{t\ge0}\phi_t(\D)$ is called the \textit{backward invariant set} of $(\phi_t)$. Any non-empty connected component of the interior of $W$ is called a \textit{petal} of $(\phi_t)$. It is easy to see that in the setting of non-elliptic groups, there exists a unique petal and it coincides with the whole unit disk. Given a petal $\Delta$, the backward orbits of all points contained in $\Delta$ converge to the same fixed point of $(\phi_t)$. Frequently, we will say that $\Delta$ corresponds to this fixed point, or the inverse. If the backward orbits contained in $\Delta$ converge to a repelling fixed point, we say that the petal is \textit{hyperbolic}. If they converge back to the Denjoy-Wolff point, then $\Delta$ is said to be \textit{parabolic}.

Assume that $(\phi_t)$ is a non-elliptic semigroup with Denjoy-Wolff point $\tau\in\partial\D$, Koenigs function $h$ and Koenigs domain $\Omega$. Let $\Delta$ be a hyperbolic petal. Then $\Delta$ corresponds to a repelling fixed point $\sigma\in\partial\D$ with repelling spectral value $\nu\in(-\infty,0)$. Then, we know (see \cite[Section 13.5]{BCDM}) that $\Delta$ is mapped through $h$ onto a maximal in $\Omega$ horizontal strip whose width is equal to $-\pi/\nu$. By maximal in $\Omega$, we mean that there exists no horizontal strip $S$ such that $h(\Delta)\subset S\subseteq \Omega$. On the contrary, if $\Delta$ is parabolic, then $h(\Delta)$ is a maximal in $\Omega$ horizontal half-plane. In any case, it is easily verified by (\ref{eq:backward definition}) and (\ref{eq:koenigs}) that each backward orbit $\widetilde{\gamma}_z$ is mapped through $h$ onto the horizontal half-line $\{h(z)-t:t\ge0\}$ which stretches to infinity in the negative direction.

\subsection{Infinitesimal Generators}
Let $(\phi_t)$ be a semigroup in $\D$. Then (see \cite[Theorem 10.1.4]{BCDM}) there exists a unique holomorphic mapping $G:\D\to\C$ such that 
    \begin{equation}\label{infinitesimal generator}
        \frac{\partial\phi_t(z)}{\partial t}=G(\phi_t(z)), \quad\text{for all }z\in\D \text{ and }t\ge0.
    \end{equation}
This unique mapping is called the \textit{infinitesimal generator} of the semigroup. Thus, $(\phi_t)$ can be thought of as the \textit{flow} of the \textit{vector field} $G$. An important result concerning the form of an infinitesimal generator is the so-called \textit{Berkson-Porta's Formula}:
\begin{theorem}
{\normalfont\cite[Theorem 10.1.10]{BCDM}}
    Let $G:\D\to\C$ be a non-constant holomorphic function. Then $G$ is the infinitesimal generator of a semigroup in $\D$ if and only if there exist a point $\tau\in\overline{\D}$ and a non-vanishing holomorphic function $p:\D\to\{w\in\C:\textup{Re}w\ge0\}$ such that
    \begin{equation}\label{berkson-porta}
        G(z)=(z-\tau)(\overline{\tau}z-1)p(z).
    \end{equation}
\end{theorem}
The point $\tau$ in the formula is actually the Denjoy-Wolff point of the semigroup generated by $G$. In addition, in case $G$ generates an elliptic semigroup whose Denjoy-Wolff point is $0$, then $p(0)$ is equal to the spectral value of the semigroup. Furthermore, if $G$ generates a semigroup which is not a group, then $p$ maps the unit disc into the right half-plane $\C_0$.

\subsection{Hyperbolic Geometry}
During the course of the proofs, we are going to need several conformally invariant quantities. In the upcoming subsections, we will provide all the necessary background with regard to these quantities.

We start with certain notions from hyperbolic geometry. For a comprehensive account of hyperbolic quantities, we refer to \cite[Chapter 5]{BCDM} and references therein. First of all, the \textit{hyperbolic metric} in the unit disk $\D$ is given by the relation
\begin{equation}\label{eq:hyp metric}
\lambda_{\D}(z)|dz|=\frac{|dz|}{1-|z|^2}, \quad z\in\D,
\end{equation}
where the function $\lambda_\D$ is called the \textit{hyperbolic density} of $\D$. Through the hyperbolic metric, we are able to define the \textit{hyperbolic distance} $d_\D(z_1,z_2)$ in the unit disk between two points $z_1,z_2\in\D$. Indeed,

$$d_\D(z_1,z_2)=\inf\limits_{\gamma}\int\limits_{\gamma}\lambda_\D(\zeta)|d\zeta|,$$
where the infimum is taken over all piecewise smooth curves $\gamma:[0,1]\to\D$ satisfying $\gamma(0)=z_1$ and $\gamma(1)=z_2$. A useful equality stemming from examining the above infimum is
\begin{equation}\label{hyperbolic distance in the unit disk}
    d_\D(0,z)=\frac{1}{2}\log\frac{1+|z|}{1-|z|}, \quad z\in\D.
\end{equation}

Through the use of conformal mappings, we may transcend all the aforementioned notions into the setting of an arbitrary simply connected domain $\Omega\subsetneq\C$. Let $f:\Omega\to\D$ be a Riemann mapping of $\Omega$. Then, the hyperbolic density and hyperbolic distance with respect to $\Omega$ may be defined via
$$\lambda_\Omega(w)=\lambda_\D(f(w))|f'(w)|, \;\;\; d_\Omega(w_1,w_2)=d_\D(f(w_1),f(w_2)), \quad w_1,w_2\in\Omega.$$
It can be proved that the definitions above are independent of the initial choice of the Riemann mapping $f$.

A direct corollary of the definition of $d_\Omega$ is the property of conformal invariance that the hyperbolic distance possesses. In addition, the hyperbolic distance satisfies the triangle inequality. Another property is its domain monotonicity. To be specific, let $\Omega_1\subset\Omega_2\subsetneq\C$ be two simply connected domains. Then, for all $w_1,w_2\in\Omega_1$, it is true that
$$d_{\Omega_1}(w_1,w_2)\ge d_{\Omega_2}(w_1,w_2).$$

To end the subsection, we will state a useful inequality correlating hyperbolic distance and Euclidean distance. Given a simply connected domain $\Omega\subsetneq\C$ and $w\in\Omega$, we use the notation $\delta_\Omega(w)=\textup{dist}(w,\partial\Omega)$.

\begin{theorem}
{\normalfont\cite[Theorem 5.3.1]{BCDM}}
\label{DistanceLemma}
Let $\Omega\subsetneq\C$ be a simply connected domain. Then for every $w_1,w_2\in\Omega$, 
\begin{equation}\label{distance lemma equation}
\frac{1}{4}\log\left(1+\frac{|w_1-w_2|}{\min\{\delta_{\Omega}(w_1),\delta_{\Omega}(w_2)\}}\right)\le d_\Omega(w_1,w_2)\le\int\limits_{\Gamma}\frac{|dw|}{\delta_\Omega(w)},
\end{equation}
where $\Gamma$ is any piecewise $C^1$-smooth curve in $\Omega$ joining $w_1$ to $w_2$.
\end{theorem}

\subsection{Harmonic Measure}
We move on to a second conformal invariant which is going to play a pivotal role towards the goal of obtaining the desired rates of convergence. As before, let $\Omega\subsetneq\C$ be a simply connected domain and let $E$ be a Borel subset of the boundary $\partial\Omega$. The \textit{harmonic measure} of $E$ with respect to $\Omega$ is the solution of the generalized Dirichlet problem for the Laplacian in $\Omega$ with boundary values equal to $1$ on $E$ and to $0$ on $\partial\Omega\setminus E$. Given $z\in\Omega$, we will use the notation $\omega(z,E,\Omega)$ for the corresponding harmonic measure. On top of that, because of the initial Dirichlet problem, for $\zeta\in\partial\Omega$, we set $\omega(\zeta,E,\Omega)=1$ whenever $\zeta\in E$, and $\omega(\zeta,E,\Omega)=0$ whenever $\zeta\in\partial\Omega\setminus E$. Note that this extension to the boundary is not continuous in general. For a detailed presentation of the deep theory of harmonic measure, we refer the interested reader to \cite[Chapter 7]{BCDM}, \cite{GM} and \cite{Ransford}.

As we already mentioned, harmonic measure is conformally invariant. Moreover, by the definition itself, $\omega(\cdot,E,\Omega)$ is a harmonic function on $\Omega$ for every choice of $E\subset\partial\Omega$ Borel. Furthermore, $\omega(z,\cdot,\Omega)$ acts as a Borel probability measure on $\partial\Omega$, for all $z\in\Omega$. As such, the harmonic measure is always bounded below by $0$ and above by $1$. Another important piece of information is that the harmonic measure satisfies a property of domain monotonicity, as well. In particular, let $\Omega_1\subset\Omega_2\subsetneq\C$ be two simply connected domains and let $E\subset\partial\Omega_1\cap\partial\Omega_2$ be Borel. Then
$$\omega(z,E,\Omega_1)\le\omega(z,E,\Omega_2), \quad \text{for all }z\in\Omega_1.$$

Next, we proceed to some more elaborate results and estimations concerning harmonic measure. First of all, the above monotonicity property can be modified into a precise equality by means of the \textit{Strong Markov Property}. Indeed, let $\Omega_1,\Omega_2$ and $E$ be as above. Then, for all $z\in\Omega_1$, by \cite[p. 88]{Port-Stone}, 
\begin{equation}\label{markov property}
\omega(z,E,\Omega_2)=\omega(z,E,\Omega_1)+\int\limits_{\partial\Omega_1\setminus\partial\Omega_2}\omega(\zeta,E,\Omega_2)\cdot\omega(z,d\zeta,\Omega_1).
\end{equation}

We are also going to use two classical estimates about the harmonic measure in the unit disk. The first one is the so-called \textit{Beurling-Nevanlinna Projection Theorem}.
\begin{theorem}
{\normalfont\cite[Theorem 9.2]{GM}}
\label{BeurlingProjection}
    Let $E\subset\overline{\D}\setminus\{0\}$ and set $E^{*}=\{|z|:z\in E\}$. Suppose that $\D\setminus E$ remains a domain. Then
    \begin{equation}\label{eq:projection}
    \omega(z,E,\D\setminus E)\ge\omega(-|z|,E^{*},\D\setminus E^{*}), \quad\textup{for all }z\in\D\setminus E.
    \end{equation}
\end{theorem}

In connection with the above result, it is useful to know (cf. \cite[Lemma 4.5.8]{Ransford}) that for all $r \in (0,1)$, 
\begin{equation}
\label{eq:slitted-disk}
\omega(0,[r,1),\D \setminus [r,1)) = \frac{2}{\pi}\arcsin\frac{1-r}{1+r}.
\end{equation}

The second one provides a connection between the harmonic measure of a compact connected set and its diameter. This very result will allow us to express the rates of the orbits in terms of harmonic measure.

\begin{lemma}
{\normalfont\cite[Theorem 2]{FRW}\cite{Solynin}}
\label{Solynin}
    Let $E\subset\overline{\D}\setminus\{0\}$ be a compact connected set and write $d:=\textup{diam}E$. Denote by $D$ the connected component of $\D\setminus E$ that contains $0$. Let $E_d$ be an arc on $\partial\D$ satisfying $\textup{diam}[E_d]=d$ (in the extremal case when $d=2$, we take $E_d$ to be a half-circle). Then
    \begin{equation}\label{eq:solynin}
    \omega(0,E,D)\ge\omega(0,E_d,\D).
    \end{equation}
\end{lemma}

\subsection{Extremal Length and Extremal Distance}
We proceed to a final conformally invariant quantity. Consider a domain $\Omega \subsetneq \C$ and a non-negative Borel measurable function $\rho$ defined on $\Omega$. Given a rectifiable curve $\gamma$ on $\Omega$, it is possible to define its \textit{length with respect to} $\rho$, that is,
$$L(\gamma,\rho) = \int_{\gamma}\rho(z)\abs{dz}.$$
If $\Gamma$ is a family of rectifiable curves, we define its \textit{length with respect to} $\rho$ as
$$L(\Gamma,\rho) = \inf_{\gamma \in \Gamma}L(\gamma,\rho).$$
Similarly, the \textit{area with respect to} $\rho$ of $\Omega$ is
$$A(\Omega,\rho) = \iint_{\Omega}\rho(z)^2dxdy, \quad z=x+iy.$$
Through these notions, one can define the \textit{extremal length} of $\Gamma$, that is,
$$\lambda_{\Omega}(\Gamma) = \sup_{\rho}\left\lbrace \dfrac{L(\Gamma,\rho)^2}{A(\Omega,\rho)} : A(\Omega,\rho) > 0\right\rbrace,$$
where the supremum is taken over all non-negative Borel measurable functions $\rho$ defined on $\Omega$.

One of the main applications of extremal length is that it allows the definition of a conformally invariant notion of distance, known as \textit{extremal distance}. If $E,F \subset \overline{\Omega}$, the extremal distance $\lambda_{\Omega}(E,F)$ between $E$ and $F$ inside $\Omega$ is the extremal length of the family of rectifiable curves that join $E$ and $F$ inside $\Omega$.

For a deeper introduction to this topic, we refer to \cite[Chapter 4]{Beliaev} and \cite[Chapter IV]{GM}. In particular, we will make use of the following:
\begin{example}
\label{ex:rectangle}
\cite[Subsection 4.2.1]{Beliaev}
\cite[Chapter IV, Example 1.1]{GM}
Consider the rectangle $\Omega = [0,a] \times [0,b]$ for some $a,b > 0$, and its vertical sides $E = \{0\} \times [0,b]$ and $F = \{a\} \times [0,b]$. The extremal distance between them is $\lambda_{\Omega}(E,F) = a/b$.
\end{example}
Apart from other simple geometric situations, the explicit calculation of an extremal distance might be difficult. However, there are some useful rules such as the following one, which is called the \textit{Serial Rule}:
\begin{lemma}
\label{lemma:serial-rule}
{\normalfont\cite[Proposition 4.6]{Beliaev}
\cite[Chapter IV, Section 3.2]{GM}}
Let $\Omega \subsetneq \C$ be a domain. Suppose that there exists a crosscut $C$ of $\Omega$ that divides this domain into two domains $\Omega_1$ and $\Omega_2$ such that $E \subset \overline{\Omega_1}$ and $F \subset \overline{\Omega_2}$. Then,
\begin{equation}\label{eq:serial}
\lambda_{\Omega}(E,F) \geq \lambda_{\Omega_1}(E,C) + \lambda_{\Omega_2}(C,F).
\end{equation}
\end{lemma}

In this article, extremal distance will be used as a tool to estimate the harmonic measure. This is based on the following result, the so-called \textit{Beurling's Estimate}:
\begin{theorem}
\label{thm:Beurling}
{\normalfont\cite[Chapter IV, Theorem 5.3]{GM}}
Let $\Omega \subset \C$ be a Jordan domain. Consider $z \in \Omega$ and suppose that $E\subset\partial\Omega$ is a finite union of arcs. Let $\gamma$ be a Jordan arc joining $z$ and $\partial\Omega\setminus E$ inside $\Omega$. Then,
\begin{equation}\label{eq:beurling}
\omega(z,E,\Omega) \leq \dfrac{8}{\pi}\exp(-\pi\lambda_{\Omega \setminus \gamma}(\gamma,E)).
\end{equation}
\end{theorem}

\section{Forward Orbits in Non-elliptic Semigroups}
We commence the principal body of our article by proving our first main results concerning the rates of convergence of forward orbits. In the first subsection we deal with the upper bound for each type of non-elliptic semigroups, while we devote the second subsection to the unified lower bound.

\subsection{Upper Bounds}
We know that given a non-elliptic semigroup $(\phi_t)$, the forward orbit can be defined for every point in $\overline{\D}$. Of course, in the case of boundary points, their forward orbits are taken in terms of non-tangential limits. Starting from a boundary point $\zeta \in \partial\D$ that is not a fixed point, there are two distinct cases. Either the whole orbit of $\zeta$ lies on $\partial\D$ or the initial part of it lies on $\partial\D$ and then there exists $t_0\ge0$ such that $\gamma_\zeta(t)=\phi_t(\zeta)\in\D$, for all $t> t_0$. It is easy to check that the first case does not apply to parabolic semigroups of zero hyperbolic step due to the geometry of the respective Koenigs domain.
\begin{proof}[Proof of Theorem \ref{thm:forward}] 
As we already mentioned in Section 2, the result trivially holds for the boundary fixed points of the semigroup, for any type of non-elliptic semigroups. All that remains is to prove our theorem about the non-fixed points of $(\phi_t)$.

(a) Suppose first that $(\phi_t)$ is parabolic of zero hyperbolic step and let $\Omega:=h(\D)$ be its Koenigs domain. Recall that $h(0)=0$. Let $z\in\overline{\D}$ such that $z$ is not a fixed point for $(\phi_t)$ and $t>2(|h(z)|+\delta_\Omega(0))$. Set $$
a_t=\{\gamma_z(s):s\in[t,+\infty)\}=\{\phi_s(z):s\in[t,+\infty)\},$$
$A_t=h(a_t)$ and $d_t=\text{diam}[a_t]$. Certainly $|\gamma_z(t)-\tau|\le d_t \le 2$. Consider $E_d$ to be an arc on the unit circle such that $\text{diam}[E_d]=d_t$ (in case $d_t=2$, we take $E_d$ to be a half-circle). Then, it is known (see for example \cite[p. 96]{Ransford}) that 
$$\omega(0,E_d,\D)=\frac{|E_d|}{2\pi}=\frac{1}{\pi}\arcsin\frac{d_t}{2},$$
where $|\cdot|$ denotes arc-length. Combining this with Lemma \ref{Solynin}, we get
$$d_t=2\sin(\pi\omega(0,E_d,\D))\le2\pi\omega(0,E_d,\D)\le2\pi\omega(0,a_t,\D\setminus a_t),$$
which in turn implies
\begin{equation}\label{rate-harmonic-relation}
    |\gamma_z(t)-\tau|\le 2\pi\omega(0,a_t,\D\setminus a_t)=2\pi\omega(0,A_t,\Omega\setminus A_t),
\end{equation}
with this last inequality resulting through the conformal invariance of the harmonic measure. 

At this point we have to make two remarks. First of all, for all the above inequalities to make sense, it is imperative that $0\notin a_t$. But the initial condition $t>2(|h(z)|+\delta_\Omega(0))$ certifies that $0\notin A_t$, and since $h(0)=0$, then $0\notin a_t$. Secondly, if $a_t$ joins two points of the unit circle, then $\D\setminus a_t$ consists of two connected components. We make the convention that we still denote by $\D\setminus a_t$ the one containing $0$. The same holds for $\Omega\setminus A_t$. Similar remarks are valid in the next cases too, so we will refrain from pointing them out again.

 So, now we need to estimate the quantity $\omega(0,A_t,\Omega\setminus A_t)$. Recall that by the definition of the Koenigs function, we have $A_t=\{h(z)+s:s\in[t,+\infty)\}$, which is just a horizontal half-line stretching to $\infty$ towards the positive direction. Consider $q=de^{i\theta}\in\partial\Omega$ to be the point of the boundary of the Koenigs domain that is closest (in terms of Euclidean distance) to $0$. In case there are more than one point $q\in\partial\Omega$ satisfying $|q|=d=\text{dist}(0,\partial\Omega)$, the choice of $q$ can be made arbitrarily. Set $L_q=\{q-s:s\ge0\}$. Observe that $\text{Re}q \leq 0$, while $\text{Re}h(z) + t > 0$, and thus $A_t\cap L_q=\emptyset$. Evidently $\Omega\setminus A_t\subset\C\setminus(A_t\cup L_q)$ because $\Omega$ is convex in the positive direction. In addition, $A_t\subset\partial(\Omega\setminus A_t)\cap\partial(\C\setminus(A_t\cup L_q))$; see Figure \ref{fig:forward-upper-0HS}. Therefore, by the monotonicity property of harmonic measure, we obtain
\begin{equation}\label{harmonic-monotonicity}
\omega(0,A_t,\Omega\setminus A_t)\le\omega(0,A_t,\C\setminus(A_t\cup L_q)).
\end{equation}

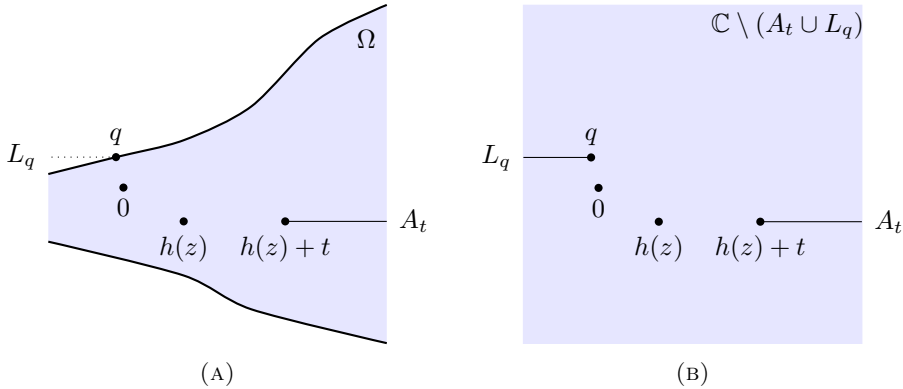
\begin{figure}[h]
\centering
\begin{subfigure}[t]{0.49\textwidth}
\centering
\begin{tikzpicture}[scale=0.9]
\fill [blue!10] (-2,-0.5) -- (-2,0.5) -- plot [smooth] coordinates {(-2,0.5) (-1, 0.75) (0,1) (1,1.5) (2,2.5) (3,3)} -- (3,3) -- (3,-2) --plot [smooth] coordinates {(3,-2) (1,-1.5) (0,-1) (-2,-0.5)};
\draw [black,thick] plot [smooth] coordinates {(-2,0.5) (-1, 0.75) (0,1) (1,1.5) (2,2.5) (3,3)};
\draw [black,thick] plot [smooth] coordinates {(-2,-0.5) (0,-1) (1,-1.5) (3,-2)};
\node[] at (2.7,2.5) {$\Omega$};
\fill[fill=black] (-0.89,0.3) circle (0.06); \node[] at (-0.89,0) {$0$};
\fill[fill=black] (-1,0.75) circle (0.06); \node[] at (-1,1.05) {$q$};
\fill[fill=black] (0,-0.2) circle (0.06); \node[] at (0,-0.6) {$h(z)$};
\fill[fill=black] (1.5,-0.2) circle (0.06); \node[] at (1.5,-0.6) {$h(z)+t$};
\draw [-] (1.5,-0.2)--(3,-0.2); \node[] at (3.4,-0.2) {$A_t$};
\draw [dotted] (-1,0.75) -- (-2,0.75); \node[] at (-2.4,0.75) {$L_q$};
\end{tikzpicture}
\caption{}
\end{subfigure}
\begin{subfigure}[t]{0.49\textwidth}
\centering
\begin{tikzpicture}[scale=0.9]
\draw [fill=blue!10,blue!10] (-2,-2) rectangle (3,3);
\node[] at (1.9,2.7) {$\mathbb{C} \setminus (A_t \cup L_q)$};
\fill[fill=black] (-0.89,0.3) circle (0.06); \node[] at (-0.89,0) {$0$};
\fill[fill=black] (-1,0.75) circle (0.06); \node[] at (-1,1.05) {$q$};
\fill[fill=black] (0,-0.2) circle (0.06); \node[] at (0,-0.6) {$h(z)$};
\fill[fill=black] (1.5,-0.2) circle (0.06); \node[] at (1.5,-0.6) {$h(z)+t$};
\draw [-] (1.5,-0.2)--(3,-0.2); \node[] at (3.4,-0.2) {$A_t$};
\draw [-] (-1,0.75)--(-2,0.75); \node[] at (-2.4,0.75) {$L_q$};
\end{tikzpicture}
\caption{}
\end{subfigure}
\caption{Construction in the proof of Theorem \ref{thm:forward}(a).}
\label{fig:forward-upper-0HS}
\end{figure}

We leave aside harmonic measure for a moment and we start working with hyperbolic distance. By Theorem \ref{DistanceLemma} and since the rectilinear segment $[0,q]$ is contained inside $\C\setminus A_t$ because $t>2(|h(z)|+d)$, we have
\begin{eqnarray}\label{hyperbolic}
 \notag d_{\C\setminus A_t}(0,q)&\le&\int\limits_{[0,q]}\frac{|dw|}{\delta_{\C\setminus A_t}(w)}=\int\limits_{0}^{d}\frac{ds} {\delta_{\C\setminus A_t}(se^{i\theta})}\\
    &\le&\int\limits_{0}^{d}\frac{ds}{|h(z)+t|-s}=\log\frac{|h(z)+t|}{|h(z)+t|-d}.
\end{eqnarray}

Using a suitable conformal mapping of $\C\setminus A_t$ onto $\D$, we may map $0$ and $q$ to $0$ and $r\in(0,1)$, respectively. In this way, the horizontal half-line $L_q$ is mapped onto a curve $\gamma_r$ in $\D$ joining $r$ to some point of the unit circle. Thus, by the conformal invariance of harmonic measure, Theorem \ref{BeurlingProjection} and relation (\ref{eq:slitted-disk}) we see that
$$\omega(0,A_t,\C\setminus(A_t\cup L_q))=\omega(0,\partial\D,\D\setminus\gamma_r)\le\omega(0,\partial\D,\D\setminus[r,1))=1-\frac{2}{\pi}\arcsin\frac{1-r}{1+r}.$$
Combining with (\ref{rate-harmonic-relation}) and (\ref{harmonic-monotonicity}), all in all, we have found that
\begin{equation}\label{rate-r}
    |\gamma_z(t)-\tau|\le 2\pi\omega(0,\partial\D,\D\setminus[r,1)).
\end{equation}
However, by the conformal invariance of the hyperbolic distance, we have that $d_{\C\setminus A_t}(0,q)=d_\D(0,r)$ which through (\ref{hyperbolic distance in the unit disk}) and (\ref{hyperbolic}) implies that
$$\log\frac{|h(z)+t|}{|h(z)+t|-d}\ge\frac{1}{2}\log\frac{1+r}{1-r},$$
or equivalently
\begin{equation}
y^2:=\left(\frac{|h(z)+t|-d}{|h(z)+t|}\right)^2\le\frac{1-r}{1+r},
\end{equation}
where $y>0$. After some straightforward calculations, we are led to
$$y^2\le\sin\left(\frac{\pi}{2}-\frac{\pi\omega(0,\partial\D,\D\setminus[r,1))}{2}\right)=\cos\frac{\pi\omega(0,\partial\D,\D\setminus[r,1))}{2}.$$
Consequently, using elementary inverse trigonometric identities
\begin{eqnarray}
\notag    \omega(0,\partial\D,\D\setminus[r,1))&\le&\frac{2}{\pi}\arccos y^2=\frac{2}{\pi}\arcsin\sqrt{1-y^4}\\
 \notag   &\le&\frac{2}{\pi}\cdot\frac{\pi}{2}\sqrt{1-y^4}\le2\sqrt{1-y}\\
    &=&2\sqrt{\frac{d}{|h(z)+t|}}\le\frac{2\sqrt{d}}{\sqrt{t-|h(z)|}}.
\end{eqnarray}
Remembering that from the start $t>2(|h(z)|+d)>2|h(z)|$ and that $d=\delta_\Omega(0)$, we return to \eqref{rate-r} to find that
\begin{equation}\label{eq:almost done}
|\gamma_z(t)-\tau|\le\frac{4\sqrt{2}\pi\sqrt{\delta_\Omega(0)}}{\sqrt{t}}.
\end{equation}
Nevertheless, by \cite[Corollary 1.4]{Pommerenke} we get that $\delta_\Omega(0)\le|h'(0)|$. As a result, (\ref{eq:almost done}) turns into
$$|\gamma_z(t)-\tau|\le4\sqrt{2}\pi\sqrt{|h'(0)|}\cdot\frac{1}{\sqrt{t}}, \quad\text{for all }t>2(|h(z)|+d).$$
Finally, for $t\in(0,2(|h(z)|+d)]$, we trivially see that
$$|\gamma_z(t)-\tau|\le2=\frac{2\sqrt{t}}{\sqrt{t}}\le 2\sqrt{2}\frac{\sqrt{|h(z)|}+\sqrt{d}}{\sqrt{t}}\le4\sqrt{2}\pi\left(\sqrt{|h(z)|}+\sqrt{|h'(0)|}\right)\cdot\frac{1}{\sqrt{t}}.$$
Combining, we obtain the desired result for all $t>0$ and all $z\in\overline{\D}$, since the initial choice of $z$ was arbitrary.

(b) Next, suppose that $(\phi_t)$ is parabolic of positive hyperbolic step and let $\Omega:=h(\D)$ be its Koenigs domain. Let $z\in\overline{\D}$. Recall that $\text{Re}h(0)=0$. Without loss of generality, we may assume that $\H$ is the smallest horizontal half-plane containing $\Omega$. Let $t>2|h(z)|$. As in the previous case, we have $A_t=\{h(z)+s:s\in[t,+\infty)\}$ and through the same arguments with the diameter estimate for harmonic measure, we get $|\gamma_z(t)-\tau|\le2\pi\omega(h(0),A_t,\Omega\setminus A_t)$. By the domain monotonicity property of harmonic measure, we also get 
\begin{equation}\label{rate-phs}
    |\gamma_z(t)-\tau|\le2\pi\omega(h(0),A_t,\H\setminus A_t).
\end{equation}
Set $A_t^{*}=\{\text{Re}h(z)+s:s\in[t,+\infty)\}\subset\partial\H$. In other words, $A_t^{*}$ is the projection of $A_t$ on the real line. Of course, in case $\gamma_z([0,+\infty))\subset\partial\D$, we have $A_t^{*}=A_t$. First suppose that this does not hold and so $A_t\ne A_t^{*}$. Obviously $A_t^{*}\subset\partial\H\cap\partial(\H\setminus A_t)$, while $\H\setminus A_t\subset\H$; see Figure \ref{fig:forward-upper-PHS-H}(A). Therefore, by the Strong Markov Property,
\begin{eqnarray}\label{eq:markov}
 \notag   \omega(h(0),A_t^{*},\H)&=&\omega(h(0),A_t^{*},\H\setminus A_t)+\int\limits_{A_t}\omega(\zeta,A_t^{*},\H)\cdot\omega(h(0),d\zeta,\H\setminus A_t)\\    \notag&\ge&\int\limits_{A_t}\omega(\zeta,A_t^{*},\H)\cdot\omega(h(0),d\zeta,\H\setminus A_t)\\
    &\ge&\min\limits_{\zeta\in A_t}\omega(\zeta,A_t^{*},\H)\cdot\omega(h(0),A_t,\H\setminus A_t).
\end{eqnarray}
By the geometry of the upper half-plane $\H$ and the slit $A_t$, it is clear that the minimum $\min_{z\in A_t}\omega(\zeta,A_t^{*},\H)$ is attained exactly on $h(z)+t$. In fact, because of symmetry, $\omega(h(z)+t,A_t^{*},\H)=\frac{1}{2}$; see Figure \ref{fig:forward-upper-PHS-H}(A). So, returning back to the Strong Markov Property and (\ref{eq:markov}), we have found that
\begin{equation}\label{markov1}
    \omega(h(0),A_t,\H\setminus A_t)\le 2\omega(h(0),A_t^*,\H).
\end{equation}

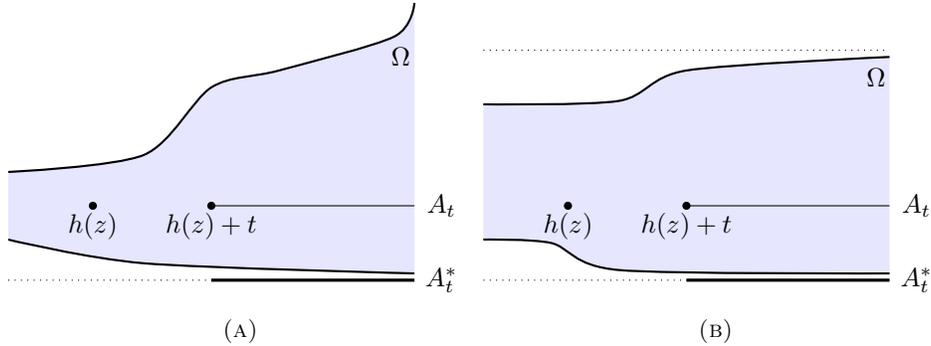
\begin{figure}[h]
\captionsetup{justification=centering}
\centering
\begin{subfigure}[t]{0.49\textwidth}
\begin{tikzpicture}[scale=0.9]
\fill [blue!10] (-2,-0.5) -- (-2,0.5) -- plot [smooth] coordinates {(-2,0.5) (0,0.75) (1,1.75) (2,2) (3.7,2.5) (4,3)} -- (4,3) -- (4,-1) --plot [smooth] coordinates {(4,-1) (0,-0.85) (-2,-0.5)};
\draw [black,thick] plot [smooth] coordinates {(-2,0.5) (0,0.75) (1,1.75) (2,2) (3.7,2.5) (4,3)};
\draw [black,thick] plot [smooth] coordinates {(-2,-0.5) (0,-0.85) (4,-1)};
\node[] at (3.8,2.2) {$\Omega$};
\fill[fill=black] (-0.75,0) circle (0.06); \node[] at (-0.75,-0.3) {$h(z)$};
\fill[fill=black] (1,0) circle (0.06); \node[] at (1,-0.3) {$h(z)+t$};
\draw [-] (1,0)--(4,0); \node[] at (4.4,0) {$A_t$};
\draw [very thick] (1,-1.1)--(4,-1.1); \node[] at (4.4,-1.1) {$A_t^*$};
\draw [dotted] (-2,-1.1)--(4,-1.1);
\end{tikzpicture}
\caption{}
\end{subfigure}
\begin{subfigure}[t]{0.49\textwidth}
\begin{tikzpicture}[scale=0.9]
\fill [blue!10] (-2,-0.5) -- (-2,1.5) -- plot [smooth] coordinates {(-2,1.5) (0,1.55) (1,2) (4,2.2)} -- (4,2.2) -- (4,-1) --plot [smooth] coordinates {(4,-1) (0,-0.95) (-1,-0.55) (-2,-0.5)};
\draw [black,thick] plot [smooth] coordinates {(-2,1.5) (0,1.55) (1,2) (4,2.2)};
\draw [black,thick] plot [smooth] coordinates {(-2,-0.5) (-1,-0.55) (0,-0.95)(4,-1)};
\node[] at (3.8,1.9) {$\Omega$};
\fill[fill=black] (-0.75,0) circle (0.06); \node[] at (-0.75,-0.3) {$h(z)$};
\fill[fill=black] (1,0) circle (0.06); \node[] at (1,-0.3) {$h(z)+t$};
\draw [-] (1,0)--(4,0); \node[] at (4.4,0) {$A_t$};
\draw [very thick] (1,-1.1)--(4,-1.1); \node[] at (4.4,-1.1) {$A_t^*$};
\draw [dotted] (-2,-1.1)--(4,-1.1);
\draw [dotted] (-2,2.3)--(4,2.3);
\end{tikzpicture}
\caption{}
\end{subfigure}
\caption{(A) Construction in the proof of Theorem \ref{thm:forward}(b), \\ (B) Construction in the proof of Theorem \ref{thm:forward}(c).}
\label{fig:forward-upper-PHS-H}
\end{figure}

Of course, if at the start $A_t=A_t^{*}$, relation (\ref{markov1}) holds trivially. Combining (\ref{rate-phs}) and (\ref{markov1}) and using the conformal invariance of the harmonic measure, in any case we get
\begin{eqnarray}\label{eq:almost done 2}
\notag    |\gamma_z(t)-\tau|&\le&4\pi\omega(h(0),A_t^{*},\H)\\
\notag    &=&4\pi\omega(h(0)-\text{Re}h(z)-t,[0,+\infty),\H)\\
\notag    &=&4\pi\frac{\pi-\arg(h(0)-\text{Re}h(z)-t)}{\pi}\\
\notag    &=&4\left(\pi-\pi+\arctan\frac{\text{Im}h(0)}{t+\text{Re}h(z)}\right)\\
    &\le&\frac{4\text{Im}h(0)}{t+\text{Re}h(z)},
\end{eqnarray}
where we have used that $\text{Re}(h(0)-\text{Re}h(z)-t) = -\text{Re}h(z)-t < 0$ since $t > 2\abs{h(z)}$. We also have that $\text{Im}h(0)=|h(0)|$ and $t+\text{Re}h(z)>\frac{t}{2}$. As a result, plugging into (\ref{eq:almost done 2}) we obtain
$$|\gamma_z(t)-\tau|\le\frac{8|h(0)|}{t}, \quad\text{for all }t>2|h(z)|.$$
Finally, whenever $t\in(0,2|h(z)|]$, we may easily compute that
$$|\gamma_z(t)-\tau|\le2=\frac{2t}{t}\le\frac{4|h(z)|}{t}.$$
So, all in all, we have
$$|\gamma_z(t)-\tau|\le\frac{8\max\{|h(0)|,|h(z)|\}}{t}, \quad\text{for all }t>0.$$

(c) Finally, we proceed to the case where $(\phi_t)$ is hyperbolic. Let $\Omega:=h(\D)$ be the Koenigs domain of the semigroup. Recall that $S = \{w\in \C : 0 < \text{Im}w < \pi/\lambda\}$, where $\lambda$ is the spectral value of $(\phi_t)$, is the smallest horizontal strip containing $\Omega=h(\D)$ and that $\text{Re}h(0)=0$. Fix $z \in \overline{\D}$ and let $t>\log2/\lambda-\text{Re}h(z)$. As usual, denote $A_t=\{h(z) + s : s\in[t,+\infty)\}$. As before,
\begin{equation}
\abs{\gamma_z(t) - \tau} \leq 2\pi \omega(h(0),A_t,\Omega \setminus A_t) \leq 2\pi \omega(h(0),A_t,S \setminus A_t).
\end{equation}
If $\mathrm{Im}h(0) \leq \pi/(2\lambda)$, define the projection $A^*_t = \{\mathrm{Re}h(z)+s : s \in [t,+\infty)\}$. Otherwise, define $A^*_t = \{\mathrm{Re}h(z)+s + i\pi/\lambda: s \in [t,+\infty)\}$. In any case, $A^*_t\subset\partial S$. Evidently, if $\gamma_z([0,+\infty))\subset\partial\D$, then $A_t=A_t^{*}$; see Figure \ref{fig:forward-upper-PHS-H}(B). First assume that this is not true. By the Strong Markov Property,
\begin{equation}
\omega(h(0),A^*_t,S) = \omega(h(0),A^*_t, S \setminus A_t) + \int\limits_{A_t}\omega(\zeta,A^*_t,S)\cdot\omega(h(0),d\zeta,S \setminus A_t).
\end{equation}
Note that $\omega(h(0),A^*_t, S \setminus A_t) \geq 0$. Due to the way we defined the projection onto $\partial S$, $\omega(\zeta,A^*_t,S) \geq 1/4$ for all $\zeta \in A_t$; see Figure \ref{fig:forward-upper-PHS-H}(B). All in all, we deduce that
\begin{equation}\label{rate-hyper}
    \abs{\gamma_z(t) - \tau} \leq 8\pi\omega(h(0),A^*_t,S).
\end{equation}
On the other hand, whenever $A_t=A_t^{*}$, the last inequality is trivial since we already deduced it in the first steps of the proof with constant $2\pi$ instead of $8\pi$. So, in any case, using a chain of conformal mappings that leave the harmonic measure invariant, we may map the horizontal strip $S$ conformally onto the right half-plane $\C_0$ and write, in the case when $\text{Im}h(0)\le\pi/(2\lambda)$,
\begin{eqnarray}\label{eq:almost done 3}
\notag    \omega(h(0),A^*_t,S)&=&\omega(h(0)-\text{Re}h(z)-t,A^*_t-\text{Re}h(z)-t,S)\\
\notag    &=&\omega(-ie^{\lambda h(0)-\lambda\text{Re}h(z)-\lambda t}+i,\{is:s\in(-\infty,0]\},\C_0)\\
\notag    &=&\frac{\frac{\pi}{2}-\arg(-ie^{\lambda h(0)-\lambda\text{Re}h(z)-\lambda t}+i)}{\pi}\\
\notag    &=&\dfrac{\frac{\pi}{2}-\arctan\dfrac{1-e^{-\lambda\text{Re}h(z)-\lambda t}\cos(\lambda\text{Im}h(0))}{e^{-\lambda\text{Re}h(z)-\lambda t}\sin(\lambda\text{Im}h(0))}}{\pi}\\
\notag    &=&\frac{1}{\pi}\arctan\dfrac{e^{-\lambda\text{Re}h(z)-\lambda t}\sin(\lambda\text{Im}h(0))}{1-e^{-\lambda\text{Re}h(z)-\lambda t}\cos(\lambda\text{Im}h(0))}\\
\notag    &\le&\frac{1}{\pi}\frac{e^{-\lambda\text{Re}h(z)-\lambda t}}{1-e^{-\lambda\text{Re}h(z)-\lambda t}}\\
    &\le&\frac{2}{\pi}e^{-\lambda\text{Re}h(z)}e^{-\lambda t},
\end{eqnarray}
where we have used the fact that $t>\log2/\lambda-\text{Re}h(z)$ several times. Combining (\ref{eq:almost done 3}) with (\ref{rate-hyper}), we obtain
$$|\gamma_z(t)-\tau|\le16e^{-\lambda\text{Re}h(z)}e^{-\lambda t}, \quad\text{for all }t>\dfrac{\log2}{\lambda}-\text{Re}h(z).$$
A similar procedure provides the same result if $\text{Im}h(0)>\pi/(2\lambda)$.
Finally, for $t\in[0,\log2/\lambda-\text{Re}h(z)$] (at least whenever this number is non-negative), in trivial fashion we find
$$|\gamma_z(t)-\tau|\le 2=2e^{\lambda t}e^{-\lambda t}\le4e^{-\lambda\text{Re}h(z)}e^{-\lambda t}.$$
Combining, we deduce the desired result for all $z\in\overline{\D}$ and all $t\ge0$.
\end{proof}
\subsection{Lower Bound}

We continue with the corresponding lower bound for the rate of convergence to the Denjoy-Wolff point. Given a non-elliptic semigroup $(\phi_t)$, we know that its Koenigs domain $\Omega$ is convex in the positive direction. Hence, it is easily observed that $\Omega$ contains horizontal half-strips that stretch to infinity in the positive direction and whose width may be enlarged as we move it towards the right. In particular, if $(\phi_t)$ is hyperbolic, then this width is always bounded above by $\pi/\lambda$, where $\lambda>0$ is the spectral value of the semigroup. This is because $\Omega$ is necessarily contained inside a minimal horizontal strip of width $\pi/\lambda$ (see Section 3). On the other hand, if $(\phi_t)$ is parabolic, then the width of the half-strips can get arbitrarily large (this time, $\lambda=0$). The inclusion of horizontal half-strips inside $\Omega$ will be the principal mechanism through which we are going to prove Theorem \ref{thm:lower}.

\begin{proof}[Proof of Theorem \ref{thm:lower}] Let $h$ be the Koenigs function of $(\phi_t)$ and $\Omega:=h(\D)$ its Koenigs domain. For the initially picked $z$, let $\ell_t$ be the length of the connected component of the set $\{w \in \C : \mathrm{Re}w = \mathrm{Re}h(z)+t\}\cap\Omega$ that contains $z$. Note that $\ell_t$ is a non-decreasing function of $t$, and $\lim_{t \to +\infty}\ell_t = \pi/\lambda$ (where we allow $\pi/0 = +\infty$ if $(\phi_t)$ is parabolic), since $\Omega$ is convex in the positive direction. Let $\epsilon > 0$, and find $T > 0$ such that $\ell_T \geq \pi/(\lambda + \epsilon)$. In that case, we can also find a horizontal half-strip $S_T \subset \Omega$ that stretches to the right, whose width is
\begin{equation}\label{eq:width}
d_T = \dfrac{\pi}{\lambda + \epsilon},
\end{equation}
and such that $h(z)+T$ is on the vertical side of $S_T$; see Figure \ref{fig:forward-lower}. Using the formula for the hyperbolic distance in the unit disk, the triangle inequality and the conformal invariance of the hyperbolic distance, we have
\begin{eqnarray}\label{LB1}
     \notag   |\phi_t(z)-\tau|&\ge&\frac{1-|\phi_t(z)|}{1+|\phi_t(z)|}\\
     \notag   &=&e^{-2d_\D(0,\phi_t(z))}\\
     \notag   &\ge&\frac{1-|z|}{1+|z|}e^{-2d_\D(z,\phi_t(z))}\\
     &=&\frac{1-|z|}{1+|z|}e^{-2d_\Omega(h(z),h(z)+t)},
\end{eqnarray}
for all $t\ge0$. Now, our intention is to estimate the quantity $d_\Omega(h(z),h(z)+t)$ and find an appropriate upper bound. By the triangle inequality of the hyperbolic distance, we have
\begin{equation}\label{LB2} 
d_\Omega(h(z),h(z)+t)\le d_\Omega(h(z),h(z)+T+1)+d_\Omega(h(z)+T+1,h(z)+t)
\end{equation}
for all $t>T+1$; see Figure \ref{fig:forward-lower}. Using Theorem \ref{DistanceLemma} and the distortion result from \cite[Corollary 1.4]{Pommerenke}, we have
\begin{equation}\label{LB3}
        d_\Omega(h(z),h(z)+T+1)\le\int\limits_{0}^{T+1}\frac{ds}{\delta_\Omega(h(z)+s)}\le\frac{T+1}{\delta_\Omega(h(z))}\le\frac{4(T+1)}{(1-|z|)|h'(z)|}.
\end{equation}
Next, let $y_T\in\R$ be the number such that the horizontal line $\{w\in\C:\text{Im}w=y_T\}$ is an axis of symmetry of $S_T$. Again, by the triangle inequality and the domain monotonicity property of the hyperbolic distance, we have
    \begin{align*}
       &d_{\Omega}(h(z)+T+1,h(z)+t)\le d_{S_T}(h(z)+T+1,\text{Re}h(z)+T+1+iy_T)+\\
       &+d_{S_T}(h(z)+t,\text{Re}h(z)+t+iy_T)+d_{S_T}(\text{Re}h(z)+T+1+iy_T,\text{Re}h(z)+t+iy_T).
    \end{align*}
    
\begin{figure}[h]
\centering
\begin{tikzpicture}[scale=1]
\fill [blue!10] (-3,-0.5) -- (-3,0.5) -- plot [smooth] coordinates {(-3,0.5) (0,2) (6,3)} -- (6,3) -- (6,-1) --plot [smooth] coordinates {(6,-1) (0,-0.85) (-3,-0.5)};
\draw [black,thick] plot [smooth] coordinates {(-3,0.5) (0,2) (6,3)};
\draw [black,thick] plot [smooth] coordinates {(-3,-0.5) (0,-0.85) (6,-1)};
\node[] at (5.8,2.7) {$\Omega$};
\node[] at (0.3,1.7) {$S_T$};
\fill[fill=black] (-2.5,0) circle (0.06); \node[] at (-2.5,-0.3) {$h(z)$};
\fill[fill=black] (0,0) circle (0.06); \node[] at (-0.8,0) {$h(z)+T$};
\fill[fill=black] (2,0) circle (0.06); \node[] at (2,-0.3) {$h(z)+T+1$};
\fill[fill=black] (5,0) circle (0.06); \node[] at (5,-0.3) {$h(z)+t$};
\fill[fill=black] (2,0.575) circle (0.06);
\fill[fill=black] (5,0.575) circle (0.06);
\draw [dashed] (0,-0.85) -- (0,2);
\draw [dashed] (0,-0.85) -- (6,-0.85);
\draw [dashed] (0,2) -- (6,2);
\draw [dash pattern={on 0.5pt off 2.5pt}] (0,0.575) -- (2,0.575); \draw [dash pattern={on 0.5pt off 2.5pt}] (5,0.575) -- (6,0.575);\node[] at (6.9,0.575) {$\mathrm{Im}z = y_T$};
\draw [dash pattern={on 1.5pt off 0.75pt}] (2,0) -- (2,0.575);
\draw [dash pattern={on 1.5pt off 0.75pt}] (5,0) -- (5,0.575);
\draw [dash pattern={on 1.5pt off 0.75pt}] (2,0.575) -- (5,0.575);
\draw [dash pattern={on 0.5pt off 2.5pt}] (0,0) -- (5,0);
\draw [-] (5,0)--(6,0); \node[] at (6.4,0) {$A_t$};
\end{tikzpicture}
\caption{Construction in the proof of Theorem \ref{thm:lower}.}
\label{fig:forward-lower}
\end{figure}
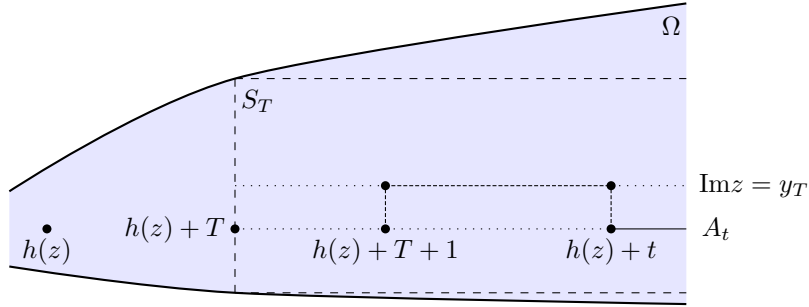

However, for any $s\in[T+1,t]$, it is easy to calculate through Theorem \ref{DistanceLemma} that
\begin{eqnarray}\label{eq:vertical distance}
 \notag   d_{S_T}(h(z)+s,\text{Re}h(z)+s+iy_T)&\le&\int\limits_{\text{Im}h(z)}^{y_T}\frac{|dx|}{\delta_{S_T}(h(z)+T+1)}\\
    &\le&\frac{\frac{\pi}{2(\lambda+\epsilon)}}{\delta_{S_T}(h(z)+T+1)}.
\end{eqnarray}
Clearly, the Euclidean distance of $h(z)+T+1$ from the vertical side of $S_T$ is $1$. Moreover, the method by which we defined $T$ combined with the convexity of $\Omega$ in the positive direction means that the Euclidean distance of $h(z)+T+1$ from the horizontal sides of $S_T$ is greater than or equal to $\delta_\Omega(h(z))$. Hence 
$$
\delta_{S_T}(h(z)+T+1)\ge\min\{1,\delta_\Omega(h(z))\}\ge\min\left\{1,\frac{1}{4}(1-|z|)|h'(z)|\right\}.$$
 So, substituting into (\ref{eq:vertical distance}), we get
\begin{equation}\label{eq:vertical 2}
d_{S_T}(h(z)+s,\text{Re}h(z)+s+iy_T)\le\frac{2\frac{\pi}{\lambda+\epsilon}}{\min\{4,(1-|z|)|h'(z)|\}},
\end{equation}
for all $s\in[T+1,t]$. Combining, we find
\begin{align}\label{LB4}
\notag d_\Omega(h(z)+T+1,h(z)+t)&\le\frac{4\frac{\pi}{\lambda+\epsilon}}{\min\{4,(1-|z|)|h'(z)|\}}+\\
&d_{S_T}(\text{Re}h(z)+T+1+iy_T,\text{Re}h(z)+t+iy_T),
\end{align}
for all $t>T+1$. Using relations (\ref{LB3}), (\ref{LB4}) and plugging them into (\ref{LB2}), we deduce
\begin{equation}\label{LB5}
        d_\Omega(h(z),h(z)+t)\le\frac{4(T+1+\frac{\pi}{\lambda+\epsilon})}{\min\{4,(1-|z|)|h'(z)|\}}+d_{S_T}(\text{Re}h(z)+T+1+iy_T,\text{Re}h(z)+t+iy_T).
\end{equation}
For the sake of brevity set $c_{z,\epsilon}=\frac{4(T+1+\frac{\pi}{\lambda+\epsilon})}{\min\{4,(1-|z|)|h'(z)|\}}$. Finally, it is easy to verify that the conformal mapping 
\begin{eqnarray*}
	z\mapsto -i\sin\left(i\frac{\pi}{d_T}(z-iy_T-\text{Re}h(z)-T)\right)\\=-i\sin\left(i(\lambda+\epsilon)(z-iy_T-\text{Re}h(z)-T)\right)
	\end{eqnarray*}
maps the horizontal half-strip $S_T$ conformally onto the right half-plane $\C_0$. Therefore, by the conformal invariance of the hyperbolic distance and its known formula for the right half-plane, we may compute that
\begin{align}\label{LB6}
    \notag   d_{S_T}(\text{Re}h(z)+T+1+iy_T&,\text{Re}h(z)+t+iy_T)=\\
    \notag &=d_{\C_0}(-i\sin(i(\lambda+\epsilon)),-i\sin(i(\lambda+\epsilon)(t-T)))\\
    \notag &=d_{\C_0}(\sinh(\lambda+\epsilon),\sinh((\lambda+\epsilon)(t-T)))\\
    \notag &=\frac{1}{2}\log\frac{\sinh((\lambda+\epsilon)(t-T))}{\sinh(\lambda+\epsilon)}\\
     &\le\frac{1}{2}\log e^{(\lambda+\epsilon)(t-T)}-\frac{1}{2}\log(2\sinh(\lambda+\epsilon)).
\end{align}

Therefore, substituting relations (\ref{LB5}) and (\ref{LB6}) into inequality (\ref{LB1}), we find

\begin{eqnarray*}
        |\gamma_z(t)-\tau|&\ge&2\frac{1-|z|}{1+|z|}e^{-2c_{z,\epsilon}} \sinh(\lambda) e^{T(\lambda+\epsilon)} e^{-(\lambda+\epsilon)t}\\
        &\ge&(1-|z|)\sinh(\lambda) e^{-2c_{z,\epsilon}} e^{-(\lambda+\epsilon)t},
\end{eqnarray*}
for all $t>T+1$. To end the proof, let $t\in[0,T+1]$. Then, working as above, we get
\begin{equation}\label{eq:almost done 4}
|\gamma_z(t)-\tau|\ge\frac{1-|z|}{1+|z|}e^{-2d_\Omega(h(z),h(z)+t)},
\end{equation}
where as before, by Theorem \ref{DistanceLemma},
\begin{equation}\label{eq:almost done 5}
d_\Omega(h(z),h(z)+t)\le\int\limits_{0}^{t}\frac{ds}{\delta_\Omega(h(z)+s)}\le\frac{t}{\delta_\Omega(h(z))}.
\end{equation}
As a consequence, by (\ref{eq:almost done 4}), (\ref{eq:almost done 5}) and using \cite[Corollary 1.4]{Pommerenke}, we get
\begin{eqnarray*}
        |\gamma_z(t)-\tau|&\ge&\frac{1-|z|}{1+|z|}e^{-\frac{2t}{\delta_\Omega(h(z))}}\\
        &\ge&\frac{1-|z|}{1+|z|}\exp{\frac{-8(T+1)}{(1-|z|)|h'(z)|}}e^{-(\lambda+\epsilon)t}\\
        &\ge&\frac{1-|z|}{1+|z|}e^{-2c_{z,\epsilon}}e^{-(\lambda+\epsilon)t},
\end{eqnarray*}
for all $t\in[0,T+1]$. Setting
$$c=c(z,\epsilon)=(1-|z|)\exp{\frac{-8(T+1+\frac{\pi}{\lambda+\epsilon})}{\max\{4,(1-|z|)|h'(z)|\}}}\min\left\{\sinh\lambda,\frac{1}{1+|z|}\right\},$$
we have the desired result for all $t\ge0$.

\end{proof}

\begin{remark}\label{remark on groups, forward}
    The proofs of Theorems \ref{thm:forward} and \ref{thm:lower} are valid even if $(\phi_t)$ is a group. However, they can be substantially simplified. In addition, by \cite[Theorem 8.2.6]{BCDM} we know that groups adhere to explicit formulas. As a result, quick calculations prove that the limit $\lim_{t\to+\infty}(e^{\lambda t}|\phi_t(z)-\tau|)$ exists for a hyperbolic group of spectral value $\lambda>0$, while the limit $\lim_{t\to+\infty}(t|\phi_t(z)-\tau|)$ exists for a parabolic group. Hence, the rates may be extracted swiftly.
\end{remark}

Now that we have established the lower bounds, we will provide an example demonstrating their sharpness for parabolic semigroups. The inclusion of $\epsilon$ in the rate does not allow us to formally speak of sharpness. So, we will show that any rate faster than $e^{-\epsilon t}$, for all $\epsilon>0$, cannot be a unified lower bound. 

\begin{example} 
\label{ex:forward-orbits-lower-bound-sharpness}
Let $f \colon [0,+\infty) \to \R$ be a smooth, increasing, and concave function such that
$$\lim_{t \to +\infty}f(t) = +\infty, \qquad \lim_{t \to +\infty}\dfrac{f(t)}{t} = 0.$$
Consider the positive function $\theta \colon (0,+\infty) \to \R$ given by $\theta = \pi/f'$. Since $f$ is concave, $\theta$ is also increasing. Therefore,
$$\Omega := \{x+iy \in \C : x > 0, 0 < y < \theta(x)\}$$
is a simply connected and convex in the positive direction domain. Therefore, through a Riemann mapping $h \colon \D \to \Omega$, we may construct a non-elliptic semigroup $(\phi_t)$ given by
$$\phi_t(z) = h^{-1}(h(z)+t), \qquad t \ge 0, \, z \in \D.$$
Since $f'$ is decreasing and $f(t) \to +\infty$, and $f(t)/t \to 0$, as $t \to +\infty$, it follows that $f'(t) \to 0$, as $t \to +\infty$. Then $\theta(t) \to +\infty$, as $t \to +\infty$. This means that $\Omega$ is not contained in a horizontal strip, and so $(\phi_t)$ is a parabolic semigroup of positive hyperbolic step (notice that $\Omega$ is contained in the upper half-plane). In particular, by the shape of $\Omega$, it is clear that there is only one prime end corresponding to $\infty$ and hence, using Carath\'eodory's Theorem \cite[Theorem 4.3.1]{BCDM}, the Denjoy-Wolff point of $(\phi_t)$ is the unique point $\tau \in \partial \D$ such that $h(\tau) = \infty$.

Let us now estimate the rate of convergence of $(\phi_t)$ to its Denjoy-Wolff point. Using harmonic measure, as we have done several times, we know that
\begin{equation}\label{eq:example1}
\abs{\phi_t(z)-\tau} \leq 2\pi\omega(h(0),A_t,\Omega \setminus A_t),
\end{equation}
where $A_t = \{h(z) + s: s \in [t,+\infty)\}$. Pick $t > 0$ such that $\textup{Re}h(z)+t > \textup{Re}h(0)$. By the maximum principle for harmonic functions we see that
\begin{equation}\label{eq:example2}
\omega(h(0),A_t,\Omega \setminus A_t) \leq \omega(h(0),L_t,\Omega_t),
\end{equation}
where
$$L_t := \{\textup{Re}h(z)+t+iy : 0 < y < \theta(\textup{Re}h(z)+t)\},$$
and
$$\Omega_t := \{x+iy \in \C : 0 < x < \textup{Re}h(z)+t, \, 0 < y < \theta(x)\}.$$
Now, let $\gamma$ be the curve traced by the horizontal line joining $h(0)$ and $\partial\Omega$. Using the serial rule for extremal distance (see Lemma \ref{lemma:serial-rule} and (\ref{eq:serial})), we obtain
\begin{equation}\label{eq:example3}
\lambda_{\Omega_t \setminus \gamma}(\gamma,L_t) \geq \lambda_{\Omega^*_t}(C,L_t),
\end{equation}
where
$$C := \{\textup{Re}h(0)+iy: 0 < y < \theta(\textup{Re}h(0))\},$$
and
$$\Omega^*_t := \{x+iy \in \C : \textup{Re}h(0) < x < \textup{Re}h(z)+t, \, 0 < y < \theta(x)\}.$$
Moreover, using \cite[Chapter IV, Eq. (6.2)]{GM} we get that
$$\lambda_{\Omega^*_t}(C,L_t) \geq \int\limits_{\textup{Re}h(0)}^{\textup{Re}h(z)+t}\dfrac{ds}{\theta(s)} = \dfrac{1}{\pi}(f(\textup{Re}h(z)+t)-f(\textup{Re}h(0))) \geq \dfrac{1}{\pi}(f(t)-f(\textup{Re}h(0))).$$
Finally, using Theorem \ref{thm:Beurling} and combining (\ref{eq:example1}), (\ref{eq:example2}), and (\ref{eq:example3}), we deduce
\begin{align*}
\abs{\phi_t(z)-\tau} & \leq 2\pi\omega(h(0),L_t,\Omega_t) \leq 16\exp(-\pi\lambda_{\Omega_t\setminus\gamma}(\gamma,L_t)) \\
& \leq 16\exp(-\pi\lambda_{\Omega^*_t}(C,L_t)) \leq 16\exp(f(\textup{Re}h(0)))\exp(-f(t)),
\end{align*}
for all those $t$ such that $\textup{Re}h(z)+t>\textup{Re}h(0)$. Since $\abs{\phi_t(z)-\tau}$ and $\exp(-f(t))$ are bounded, this means that there exists $C = C(z)$ such that
$$\abs{\phi_t(z)-\tau} \leq C\exp(-f(t)), \qquad\text{for all } t > 0.$$
Hence the desired sharpness follows. Clearly, these arguments can also be adapted for the case of zero hyperbolic step by defining the domain
$$\Omega := \{x+iy \in \C : x > 0, -\theta(x)/2 < y < \theta(x)/2\}.$$
    
\end{example}

We end the section with two remarks about the rate of convergence to the Denjoy-Wolff point concerning certain more exclusive, but still important, types of non-elliptic semigroups.

\begin{remark}\label{rmk:finite shift}
   Adding to the classification of semigroups we have already explained, non-elliptic semigroups may also be separated into two essential categories. Let $(\phi_t)$ be a non-elliptic semigroup in $\D$ with Denjoy-Wolff point $\tau\in\partial\D$. We say that $(\phi_t)$ is of \textit{finite shift} if for each $z\in\D$, there exists some \textit{horodisk} $E_z$ of $\D$ centered at $\tau$ (i.e. $E_z$ is a Euclidean disk that is internally tangent to $\D$ at the point $\tau$) such that $\phi_t(z)\notin E_z$, for all $t\ge0$. Otherwise, $(\phi_t)$ is said to be of \textit{infinite shift} and every forward orbit intersects all horodisks centered at $\tau$. It can be proved (see e.g. \cite[p.511]{BCDM}) that semigroups of finite shift are necessarily parabolic of positive hyperbolic step. Therefore, if $(\phi_t)$ is a semigroup of finite shift, its rate of convergence is bounded above from $c_1/t$, where $c_1$ is the constant in Theorem \ref{thm:forward}(b). However, in \cite[Theorem 1.1(ii)]{KTZ} the authors prove that the rate in the case of finite shift is also bounded below by $c_2/t$, where $c_2$ is a positive constant that depends on the starting point. Hence, we understand that semigroups of finite shift have a rate that basically behaves like $1/t$, as $t\to+\infty$. In particular, in \cite[Theorem 4.2]{Fran} the exact value of $\lim_{t\to+\infty}(t|\phi_t(z)-\tau|)$ is computed and is shown to be independent of $z$.
\end{remark}

\begin{remark}\label{rmk:conformality Denjoy}
    In addition, hyperbolic semigroups can also be separated in two categories. Let $(\phi_t)$ be a hyperbolic semigroup in $\D$ with Denjoy-Wolff point $\tau\in\partial\D$, spectral value $\lambda>0$ and Koenigs function $h$. Set $g:=-i\exp(-\lambda h)$. Then, the separation into the two categories results from the conformality or not of $g$ at $\tau$ (see \cite[Section 4]{BCDM-Rate}). This differentiation yields an important distinction with regard to the rates of convergence. More specifically, if $g$ is indeed conformal at $\tau$, then the $\epsilon$ in the lower bound from Theorem \ref{thm:lower} may be eliminated, rendering the rate of convergence essentially equal to $e^{-\lambda t}$ because of the upper bound from Theorem \ref{thm:forward}(c). On the contrary, if $g$ is not conformal at $\tau$, then the $\epsilon$ cannot be removed and Theorem \ref{thm:lower} provides the best possible lower bound (cf. \cite[Theorem 4.2]{BCDM-Rate}). This fact also yields the sharpness of the rate in Theorem \ref{thm:lower} for hyperbolic semigroups.
\end{remark}

\section{Backward Orbits in Non-elliptic Semigroups}

Having concluded our study on the rates of the forward orbits in non-elliptic semigroups, we are now ready to proceed to the focal point of this work that concerns backward orbits and their rates of convergence. In contrast to the forward case, certain intricacies appear when working in backward dynamics. A first difficulty has to do with the fact that not all backward orbits converge to the same point. Hence, we have to study the quantity $|\widetilde{\gamma}_z(t)-\sigma|$, where $\sigma$ depends on the starting point $z$. Moreover, given $z\in\D$, the backward orbit $t \mapsto \widetilde{\gamma}_z(t)$ might not be defined for all $t \geq 0$. Furthermore, whenever $z\in\partial\D$, its backward orbit is not necessarily well-defined. For all these reasons, as usual, we will solely deal with backward orbits emanating from points $z\in\D$ satisfying $T_z=+\infty$.

\subsection{Regular Backward Orbits}
We start our study with the regular backward orbits of non-elliptic semigroups. In other words, with all those backward orbits which are contained inside the petals of the semigroup. In the following proof, we will start our research from parabolic petals which can only appear in parabolic semigroups. In most instances, the behavior of a semigroup concerning its backward orbits depends only on the type of its petals and not the type of the semigroup itself. Hence, one would expect backward orbits contained inside a parabolic petal to present a unified rate, independent of the type of $(\phi_t)$ (i.e. independent of the hyperbolic step of $(\phi_t)$). Nevertheless, this is not the case. We will see that the rates of parabolic semigroups of zero and positive hyperbolic step are indeed different. This difference will be further demonstrated later on by means of two examples of sharpness. During the course of the next proof, we uphold the normalizations and conventions made in Section 2.

\begin{proof}[Proof of Theorem \ref{thm:backward}] (a) First, suppose that $(\phi_t)$ is parabolic of zero hyperbolic step and that $\Delta$ is a parabolic petal of $(\phi_t)$ with $z\in\Delta$. Then $\widetilde{\gamma}_z(t)$ is defined for all $t\ge0$ and in particular $\lim_{t\to+\infty}\widetilde{\gamma}_z(t)=\tau$. We start with the upper bound. Let $t>2(|h(z)|+\delta_\Omega(0))$, where $\Omega:=h(\D)$ is the Koenigs domain of $(\phi_t)$. Set $\beta_t=\{\widetilde{\gamma}_z(s):s\in[t,+\infty)\}$ and $B_t=h(\beta_t)=\{h(z)-s:s\in[t,+\infty)\}$. Recall that $h(0)=0$. Passing through the diameter estimate we used before, we can see that
$$|\widetilde{\gamma}_z(t)-\tau|\le2\pi\omega(0,\beta_t,\D\setminus\beta_t)=2\pi\omega(0,B_t,\Omega\setminus B_t).$$
Following similar steps as in the proof of Theorem \ref{thm:forward}(a), we find $q\in\partial\Omega$ such that $|q|=|q-0|=\delta_\Omega(0)$ and prove that
\begin{equation}\label{hyper-dist-back}
   d_{\C\setminus B_t}(0,q)\le\log\frac{|h(z)-t|}{|h(z)-t|-\delta_\Omega(0)}. 
\end{equation}
Then, denoting $L_q:=\{q-s:s\in[0,+\infty)\}$ and using a suitable conformal mapping onto the unit disk, we may show that
\begin{equation}\label{eq:double slit}
\omega(0,B_t,\C\setminus(B_t\cup L_q))\le1-\frac{2}{\pi}\arcsin\frac{1-r}{1+r},
\end{equation}
for some $r\in(0,1)$ satisfying $d_{\C\setminus B_t}(0,q)=d_\D(0,r)$. For this $r$, via relation (\ref{hyper-dist-back}), we get
\begin{equation}\label{eq:log}
\frac{1}{2}\log\frac{1+r}{1-r}\le\log\frac{|h(z)-t|}{|h(z)-t|-\delta_\Omega(0)}.
\end{equation}
Combining relations (\ref{eq:double slit}), (\ref{eq:log}) with \cite[Corollary 1.4]{Pommerenke} and executing identical calculations as in the proof of Theorem \ref{thm:forward}(a), we find
$$|\widetilde{\gamma}_z(t)-\tau|\le4\sqrt{2}\pi\sqrt{|h'(0)|}\cdot\frac{1}{\sqrt{t}}, \quad\text{for all }t>2(|h(z)|+\delta_\Omega(0)).$$
Finally, for $t\in(0,2(|h(z)|+\delta_\Omega(0))]$, it is elementary to see that
$$|\widetilde{\gamma}_z(t)-\tau|\le2=\frac{2\sqrt{t}}{\sqrt{t}}\le4\sqrt{2}\pi\left(\sqrt{|h(z)|}+\sqrt{|h'(0)|}\right)\cdot\frac{1}{\sqrt{t}}.$$
Hence, we have the desired result for all $t>0$ and all $z\in\Delta$.

We proceed to the lower bound. First of all, notice that by Julia's Lemma, the backward orbit $\widetilde{\gamma}_z$ stays out of a horodisk, which is determined by the initial point $z$. This means that
\begin{equation}
\abs{\widetilde{\gamma}_z(t) - \tau}^2 \geq \dfrac{\abs{z-\tau}^2}{1-\abs{z}^2}(1-\abs{\widetilde{\gamma}_z(t)}^2) \geq  \dfrac{\abs{z-\tau}^2}{1-\abs{z}^2}(1-\abs{\widetilde{\gamma}_z(t)}),
\end{equation}
which yields
\begin{equation}\label{parabolic-lower}
    \abs{\widetilde{\gamma}_z(t)-\tau} \geq \dfrac{\abs{z-\tau}}{\sqrt{1-\abs{z}^2}}\sqrt{1-\abs{\widetilde{\gamma}_z(t)}}.
\end{equation}
The latter can be estimated through the use of hyperbolic distance:
\begin{align}\label{eq:1-||}
\notag 1-\abs{\widetilde{\gamma}_z(t)} & = \exp(-2d_{\D}(0,\widetilde{\gamma}_z(t)))(1+\abs{\widetilde{\gamma}_z(t)}) \\
\notag& \geq \exp(-2d_{\D}(0,z))\exp(-2d_{\D}(z,\widetilde{\gamma}_z(t))) \\
& = \dfrac{1-\abs{z}}{1+\abs{z}}\exp(-2d_{\D}(z,\widetilde{\gamma}_z(t))),
\end{align}
where we have used the triangle inequality for $d_{\D}$. Plugging (\ref{eq:1-||}) into (\ref{parabolic-lower}), we are led to
\begin{equation}\label{parabolic-lower2}
    |\widetilde{\gamma}_z(t)-\tau|\ge\frac{|z-\tau|}{1+|z|}e^{-d_\D(z,\widetilde{\gamma}_z(t))}.
\end{equation}
We now try to estimate the latter hyperbolic distance using 
its conformal invariance. Without loss of generality, we may assume that $h(\Delta) = \{w\in\C:\text{Im}w>\alpha\} \subset \Omega$, for some $\alpha\in\R$. In case $h(\Delta)$ is a lower half-plane, the proof follows in similar fashion. Using conformal mappings, we may write
\begin{align}\label{back, lower, hyper}
\notag d_{\D}(z,\widetilde{\gamma}_z(t)) & = d_{\Omega}(h(z),h(z)-t) \leq d_{h(\Delta)}(h(z),h(z)-t) \\
\notag & = d_{\H}(h(z)-i\alpha,h(z)-i\alpha-t) = \frac{1}{2}\log\frac{|t+2i(\mathrm{Im}h(z)-\alpha)|+t}{|t+2i(\mathrm{Im}h(z)-\alpha)|-t} \\
&=\frac{1}{2}\log\frac{2(\mathrm{Im}h(z)-\alpha)^2+t^2+t\sqrt{4(\mathrm{Im}h(z)-\alpha)^2+t^2}}{2(\mathrm{Im}h(z)-\alpha)^2}.
\end{align}
Now let $t>|\mathrm{Im}h(z)-\alpha|$. Continuing with the previous calculations, at once, we get 
$$d_{\D}(z,\widetilde{\gamma}_z(t))\le\frac{1}{2}\log\frac{(3+\sqrt{5})t^2}{2(\mathrm{Im}h(z)-\alpha)^2}\le \frac{1}{2}\log\frac{4t^2}{(\mathrm{Im}h(z)-\alpha)^2}=\log\frac{2t}{|\mathrm{Im}h(z)-\alpha|}.$$
On the other hand, for $t\in(1,|\textup{Im}h(z)-\alpha|]$ (whenever this modulus is indeed greater than $1$), relation (\ref{back, lower, hyper}) yields
$$d_\D(z,\widetilde{\gamma}_z(t))\le\frac{1}{2}\log\frac{(3+\sqrt{5})(\textup{Im}h(z)-\alpha)^2}{2(\textup{Im}h(z)-\alpha)^2}\le\log2\le \log(2t).$$
Combining, we find $d_\D(z,\widetilde{\gamma}_z(t))\le\log\frac{2t}{\min\{1,|\textup{Im}h(z)-\alpha|\}}$, for all $t>1$. 
Thus, substituting in (\ref{parabolic-lower2}), the desired result follows.

(b) Next, suppose that $(\phi_t)$ is parabolic of positive hyperbolic step with Koenigs domain $\Omega$. This proof is the backward analogue of the proof of Theorem \ref{thm:forward}(b). Let $\Delta$ be a parabolic petal of $(\phi_t)$, $z\in\Delta$ and $t>2|h(z)|$. We maintain the notation $\beta_t$ and $B_t$ from the previous case. Recall that $\text{Re}h(0)=0$. Without loss of generality, we may assume that $\H$ is the smallest horizontal half-plane containing $\Omega$. Using consecutively the usual diameter estimate, the conformal invariance and the domain monotonicity property of harmonic measure, we find
$$|\widetilde{\gamma}_z(t)-\tau|\le2\pi\omega(0,\beta_t,\D\setminus\beta_t)=2\pi\omega(h(0),B_t,\Omega\setminus B_t)\le2\pi\omega(h(0),B_t,\H\setminus B_t).$$
Consider $B^*_t$ to be the projection of $B_t$ onto the real line $\R=\partial\H$. In other words, $B^*_t=\{\text{Re}h(z)-s:s\in[t,+\infty)\}$. Using the Strong Markov Property exactly as in the proof of Theorem \ref{thm:forward}(b), we have
\begin{equation}\label{eq:markov2}
\omega(h(0),B_t,\H\setminus B_t)\le2\omega(h(0),B^*_t,\H).
\end{equation}
Therefore, through (\ref{eq:markov2}) we deduce
\begin{equation}\label{rate-phs-petal}
    |\widetilde{\gamma}_z(t)-\tau|\le4\pi\omega(h(0),B^*_t,\H).
\end{equation}
However, the latter harmonic measure can be easily estimated. As a matter of fact,
\begin{eqnarray}\label{eq:harmonic measure half-plane}
\notag    \omega(h(0),B^*_t,\H)&=&\omega(h(0)-\text{Re}h(z)+t,(-\infty,0],\H)\\
\notag    &=&\frac{\arg(h(0)-\text{Re}h(z)+t)}{\pi}\\
\notag    &=&\frac{1}{\pi}\arctan\frac{\text{Im}h(0)}{t-\text{Re}h(z)}\\
    &\le&\frac{2|h(0)|}{\pi t},
\end{eqnarray}
where the last inequality is implied through $t>2|h(z)|$. Plugging (\ref{eq:harmonic measure half-plane}) into (\ref{rate-phs-petal}), we arrive to
$$|\widetilde{\gamma}_z(t)-\tau|\le\frac{8|h(0)|}{t}, \quad\text{for all }t>2|h(z)|.$$
To end the proof, for $t\in(0,2|h(z)|]$, we get
$$|\widetilde{\gamma}_z(t)-\tau|\le2=\frac{2t}{t}\le\frac{4|h(z)|}{t}.$$
Combining the two results, we obtain the desired bound for all $t>0$ and all $z\in\Delta$.

For the lower bound, we note that the proof of the respective lower bound in the previous case works in this case, as well. Indeed, that proof only takes under consideration the existence of a parabolic petal and not the type of the semigroup.

(c) Finally, suppose that $\Delta$ is a hyperbolic petal. Let $\Omega$ be the Koenigs domain of the semigroup. By the diameter estimate for the harmonic measure, we can see that
\begin{equation}\label{eq:last diameter estimate}
|\widetilde{\gamma}_z(t)-\sigma|\le 2\pi\omega(h(0),B_t,\Omega\setminus B_t),
\end{equation}
where $B_t=\{h(z)-s:s\in[t,+\infty)\}$. Since $\Delta$ is a hyperbolic petal, we know that $h(\Delta)$ is a maximal horizontal strip contained in $\Omega$ whose width is equal to $-\pi/\nu$. Let $\epsilon\in(0,-\nu)$. For the initially fixed $z\in\D$, the fact that it is contained in a petal, dictates that there exists no point $\zeta\in\partial\Omega$ satisfying $\text{Im}\zeta=\text{Im}h(z)$. Therefore, we are able to separate $\partial\Omega$ into two disjoint, connected (in the sense of prime ends) boundary components $\partial\Omega^+$ and $\partial\Omega^-$, with $\partial\Omega^+=\{\zeta\in\partial\Omega:\text{Im}\zeta>\text{Im}h(z)\}$ and $\partial\Omega^-=\{\zeta\in\partial\Omega:\text{Im}\zeta<\text{Im}h(z)\}$. Clearly, $\partial\Omega^+\cup\partial\Omega^-=\partial\Omega$ and $\partial\Omega^+\cap\partial\Omega^-=\emptyset$. For the initially fixed $z\in\D$ and for $t\ge0$, we denote by $L_t$ the vertical line containing $h(z)-t$, i.e.
$$L_t=\{w\in\C:\text{Re}w=\text{Re}h(z)-t\}.$$
As we already mentioned, due to the fact that $z$ is contained inside a hyperbolic petal $\Delta$ and that $h(\Delta)$ is a maximal horizontal strip in $\Omega$, it is necessary that there exists a smallest $T_z\ge0$ (which depends in an essential way on the geometry of $\Omega$ and the position of $z$) such that both the sets $L_t\cap\partial\Omega^+$ and $L_t\cap\partial\Omega^-$ are non-empty, for all $t>T_z$. The number $T_z$ itself is either an infimum or a minimum, but this does not change anything. For all those $t>T_z$, we may find the point $p_t^+\in L_t\cap\partial\Omega^+$ with the smallest imaginary part and the point $p_t^-\in L_t\cap\partial\Omega^-$ with the largest imaginary part; see Figure \ref{fig:backward-upper-hyperbolic}(A). Set $d_t=|p_t^+-p_t^-|=\text{Im}p_t^+-\text{Im}p_t^-$. Because of the convexity in the positive direction of $\Omega$, it is clear that $d_t$ is a decreasing function of $t>T_z$. In fact, $d_t\to-\pi/\nu$, as $t\to+\infty$. For the $\epsilon>0$ we chose, there exists $t_{z,\epsilon}>T_z$ such that $d_{t_{z,\epsilon}} \leq -\pi/(\nu+\epsilon)$ (for this to be true, $\epsilon$ has to be sufficiently small). We also potentially shrink $\epsilon$ so that $t_{z,\epsilon}>|h(z)-h(0)|$. Consider the horizontal half-lines
$$Q_t^+=\{p_t^+-s:s\ge0\} \quad\text{and}\quad Q_t^-=\{p_t^--s:s\ge0\}.$$
\begin{figure}[h]
\centering
\begin{subfigure}{0.49\textwidth}
\begin{tikzpicture}[scale=0.9]
\fill [blue!10] (-2,-0.5) -- (-2,0.5) -- plot [smooth] coordinates {(-2,0.5) (0,1) (1,1.5) (2,2.5) (3,3)} -- (3,3) -- (3,-2) --plot [smooth] coordinates {(3,-2) (1,-1.5) (0,-1) (-2,-0.5)};
\node[] at (2.7,2.5) {$\Omega$};
\draw [black,thick] plot [smooth] coordinates {(-2,0.5) (0,1) (1,1.5) (2,2.5) (3,3)};
\draw [black,thick] plot [smooth] coordinates {(-2,-0.5) (0,-1) (1,-1.5) (3,-2)};
\draw [dotted] (1,1.5)--(1,-1.5); 
\fill[fill=black] (1,0.2) circle (0.06); \node[] at (0,0.2) {$h(z)-t_{z,\epsilon}$};
\fill[fill=black] (1,1.5) circle (0.06); \node[] at (0.6,1.8) {$p_{t_{z,\epsilon}}^+$};
\fill[fill=black] (1,-1.5) circle (0.06); \node[] at (0.6,-1.8) {$p_{t_{z,\epsilon}}^-$};
\fill[fill=black] (2.9,0.2) circle (0.06); \node[] at (2.4,0.2) {$h(z)$};
\fill[fill=black] (2,-1) circle (0.06); \node[] at (2.5,-1) {$h(0)$};
\draw [dotted] (1,1.5) -- (-2,1.5); 
\draw [dotted] (1,-1.5) -- (-2,-1.5); 
\end{tikzpicture}
\caption{}
\end{subfigure}
\begin{subfigure}{0.49\textwidth}
\begin{tikzpicture}[scale=0.9]
\draw [fill=blue!10,blue!10] (-5,3) rectangle (1.5,1.5);
\draw [fill=blue!10,blue!10] (-5,-3) rectangle (1.5,-1.5);
\draw [fill=blue!10,blue!10] (-4,-1.5) rectangle (1.5,1.5);
\draw [dotted] (0,1.5)--(0,-1.5); 
\draw [-] (0,1.5)--(-5,1.5); \node[] at (-2,1.75) {$Q_{t_{z,\epsilon}}^+$};
\draw [-] (0,-1.5)--(-5,-1.5); \node[] at (-2,-1.75) {$Q_{t_{z,\epsilon}}^-$};
\draw [-] (-4,1.5)--(-4,-1.5); \node[] at (-4.5,0) {$D_t$};
\fill[fill=black] (0,0.5) circle (0.06); \node[] at (-1,0.5) {$h(z)-t_{z,\epsilon}$};
\fill[fill=black] (-4,0.5) circle (0.06); \node[] at (-3,0.5) {$h(z)-t$};
\fill[fill=black] (0,1.5) circle (0.06); \node[] at (0.5,1.5) {$p_{t_{z,\epsilon}}^+$};
\fill[fill=black] (0,-1.5) circle (0.06); \node[] at (0.5,-1.5) {$p_{t_{z,\epsilon}}^-$};
\fill[fill=black] (1,-2.5) circle (0.06); \node[] at (0.5,-2.5) {$h(0)$};
\node[] at (1.25,2.75) {$A_\epsilon$};
\end{tikzpicture}
\caption{}
\end{subfigure}
\caption{Construction in the proof of Theorem \ref{thm:backward}(c).}
\label{fig:backward-upper-hyperbolic}
\end{figure}
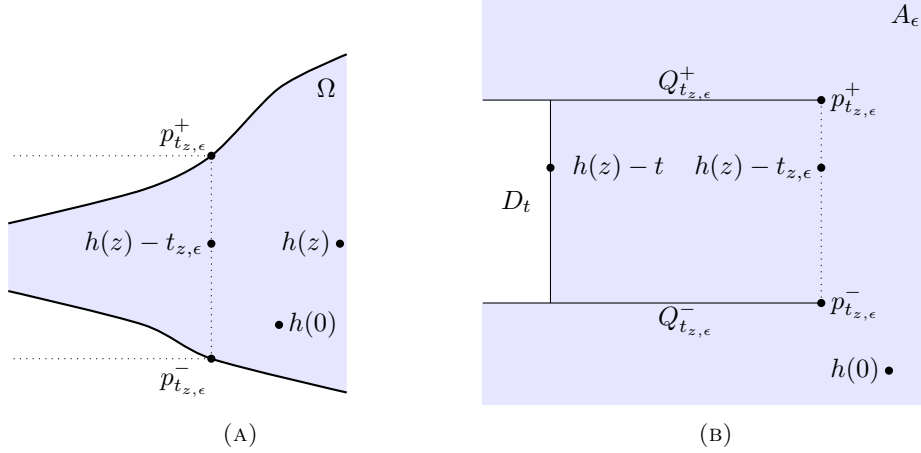
Evidently, $\Omega\subset\C\setminus(Q_t^+\cup Q_t^-)=:\Omega_t$. Therefore, 
$$
\omega(h(0),B_t,\Omega\setminus B_t)\le\omega(h(0),B_t,\Omega_t\setminus B_t),
$$
 for all $t>T_z$, by the domain monotonicity property of the harmonic measure. Of course, the same inequality holds for $t=t_{z,\epsilon}$. Fix $t>t_{z,\epsilon}$. In addition, consider the vertical rectilinear segments $D_t=\{\text{Re}h(z)-t+is:s\in[\text{Im}p_{t_{z,\epsilon}}^-,\text{Im}p_{t_{z,\epsilon}}^+]\}$. It is easy to see that $D_t$ joins the half-lines $Q_{t_{z,\epsilon}}^+$ and $Q_{t_{z,\epsilon}}^-$. By the maximum principle for harmonic functions, we have
\begin{equation}\label{eq:maximum principle}
\omega(h(0),B_t,\Omega\setminus B_t)\le\omega(h(0),B_t,\Omega_{t_{z,\epsilon}}\setminus B_t)\le\omega(h(0),D_t,A_\epsilon),
\end{equation}
where $A_\epsilon$ is the simply connected domain bounded by $Q_{t_{z,\epsilon}}^+$, $Q_{t_{z,\epsilon}}^-$, and $D_t$, but not the horizontal half-strip; see Figure \ref{fig:backward-upper-hyperbolic}(B). The last harmonic measure is well-defined, because the condition $t>t_{z,\epsilon}>|h(z)-h(0)|$ dictates that $h(0)\in A_\epsilon$. 
Certainly, the horizontal half-line $\Gamma:=\{h(0)+s:s\ge0\}$ joins $h(0)$ to $\partial A_\epsilon\setminus D_t$. Therefore, by Beurling's Estimate, we get
\begin{equation}\label{eq:last beurling}
\omega(h(0),D_t,A_\epsilon)\le\frac{8}{\pi}e^{-\pi\lambda_{A_\epsilon\setminus\Gamma}(D_t,\Gamma)}.
\end{equation}
However, by the serial rule for the extremal distance, we obtain
\begin{equation}\label{eq:last serial}
\lambda_{A_\epsilon\setminus\Gamma}(D_t,\Gamma)\ge\lambda_{R_\epsilon}(D_t,D_{t_{z,\epsilon}})=\frac{\text{Re}h(z)-t_{z,\epsilon}-(\text{Re}h(z)-t)}{d_{t_{z,\epsilon}}},
\end{equation}
where $R_\epsilon$ is the rectangle bounded by $D_t,D_{t_{z,\epsilon}},Q_{t_{z,\epsilon}}^+$, and $Q_{t_{z,\epsilon}}^-$.
Combining (\ref{eq:last diameter estimate}), (\ref{eq:maximum principle}), (\ref{eq:last beurling}) and (\ref{eq:last serial}) we have
$$|\widetilde{\gamma}_z(t)-\sigma|\le 2\pi\frac{8}{\pi}e^{-\pi\frac{t-t_{z,\epsilon}}{d_{t_{z,\epsilon}}}} \leq 16e^{-(\nu+\epsilon)t_{z,\epsilon}}e^{(\nu+\epsilon)t},$$
for all $t>t_{z,\epsilon}$. Finally, for $t\le t_{z,\epsilon}$, we have
$$|\widetilde{\gamma}_z(t)-\sigma|\le2=2e^{-(\nu+\epsilon)t}e^{(\nu+\epsilon)t}\le2e^{-(\nu+\epsilon)t_{z,\epsilon}}e^{(\nu+\epsilon)t}.$$
Setting $T=t_{z,\epsilon}$, we have the desired result for all $t>0$.

Finally, we are ready to work on the lower bound. Using consecutively the triangle inequality for the Euclidean distance, the formula of the hyperbolic distance in the unit disk, the triangle inequality for the hyperbolic distance, the conformal invariance of the hyperbolic distance and the domain monotonicity property of the hyperbolic distance, we get
\begin{eqnarray}\label{hyperbolic-lower}
    \notag    |\widetilde{\gamma}_z(t)-\sigma|&\ge&1-|\widetilde{\gamma}_z(t)|\\
    \notag    &\ge&\frac{1-|\widetilde{\gamma}_z(t)|}{1+|\widetilde{\gamma}_z(t)|}\\
    \notag    &=&e^{-2d_\D(0,\widetilde{\gamma}_z(t))}\\
    \notag    &\ge&e^{-2d_\D(0,z)}e^{-2d_\D(z,\widetilde{\gamma}_z(t))}\\
    \notag    &=&\frac{1-|z|}{1+|z|}e^{-2d_\Omega(h(z),h(z)-t)}\\
        &\ge&\frac{1-|z|}{1+|z|}e^{-2d_{h(\Delta)}(h(z),h(z)-t)},
\end{eqnarray}
since $z\in\Delta$ and $h(\Delta)\subset\Omega$. Now, we have to find an upper bound for the quantity $d_{h(\Delta)}(h(z),h(z)-t)$. Since $h(\Delta)$ is a maximal horizontal strip of width $-\pi/\nu$ inside $\Omega$, there exists some $\alpha\in\R$ so that $h(\Delta)=\{w\in\C:\alpha<\text{Im}w<\alpha-\pi/\nu\}$. It is easy to verify that the conformal mapping $z\mapsto e^{-\nu(z-i\alpha+i\frac{\pi}{2\nu})}$ maps $h(\Delta)$ conformally onto the right half-plane $\C_0$. Through the conformal invariance of the hyperbolic distance and the formula
\begin{equation}\label{eq:hyperbolic right half-plane}
d_{\C_0}(w_1,w_2)=\frac{1}{2}\log\frac{|w_1+\bar{w_2}|+|w_1-w_2|}{|w_1+\bar{w_2}|-|w_1-w_2|}, \quad\text{for all }w_1,w_2\in\C_0,
\end{equation}
we are led to
\begin{eqnarray*}
        &&d_{h(\Delta)}(h(z),h(z)-t)=d_{\C_0}(-ie^{-\nu h(z)}e^{i\alpha\nu},-ie^{-\nu h(z)}e^{i\alpha\nu}e^{\nu t})\\
        &=&\frac{1}{2}\log\frac{|-ie^{-\nu h(z)}e^{i\alpha\nu}+ie^{-\nu\overline{h(z)}}e^{-i\alpha\nu}e^{\nu t}|+|-ie^{-\nu h(z)}e^{i\alpha\nu}+ie^{-\nu h(z)}e^{i\alpha\nu}e^{\nu t}|}{|-ie^{-\nu h(z)}e^{i\alpha\nu}+ie^{-\nu\overline{h(z)}}e^{-i\alpha\nu}e^{\nu t}|-|-ie^{-\nu h(z)}e^{i\alpha\nu}+ie^{-\nu h(z)}e^{i\alpha\nu}e^{\nu t}|}\\
        &=&\frac{1}{2}\log\frac{|1-e^{2i\nu\text{Im}h(z)}e^{-2i\alpha\nu}e^{\nu t}|+|1-e^{\nu t}|}{|1-e^{2i\nu\text{Im}h(z)}e^{-2i\alpha\nu}e^{\nu t}|-|1-e^{\nu t}|}\\
        &=&\frac{1}{2}\log\frac{\sqrt{1-2e^{\nu t}\cos(2\nu\text{Im}h(z)-2\alpha\nu)+e^{2\nu t}}+(1-e^{\nu t})}{\sqrt{1-2e^{\nu t}\cos(2\nu\text{Im}h(z)-2\alpha\nu)+e^{2\nu t}}-(1-e^{\nu t})}\\
        &=&\frac{1}{2}\log\frac{(\sqrt{1-2e^{\nu t}\cos(2\nu\text{Im}h(z)-2\alpha\nu)+e^{2\nu t}}+(1-e^{\nu t}))^2}{1-2e^{\nu t}\cos(2\nu\text{Im}h(z)-2\alpha\nu)+e^{2\nu t}-1+2e^{\nu t}-e^{2\nu t}}\\
        &\le&\frac{1}{2}\log\frac{(\sqrt{1+2e^{\nu t}+e^{2\nu t}}+(1-e^{\nu t}))^2}{2e^{\nu t}(1-\cos(2\nu\text{Im}h(z)-2\alpha\nu))}\\
        &=&\frac{1}{2}\log\frac{2e^{-\nu t}}{1-\cos(2\nu\text{Im}h(z)-2\alpha\nu)}.
\end{eqnarray*}
Hence, returning to (\ref{hyperbolic-lower}), we find
$$|\widetilde{\gamma}_z(t)-\sigma|\ge\frac{1-|z|}{1+|z|}\exp{\left(-\log\frac{2e^{-\nu t}}{1-\cos(2\nu\text{Im}h(z)-2\alpha\nu)}\right)},$$
which leads to the desired result at once.
\end{proof}

\begin{remark}\label{remark1}
Inspecting the proofs for the upper bounds in Theorem \ref{thm:backward}(a) and (b) carefully, one may observe that the inclusion of $z$ in the parabolic petal $\Delta$ did not play any significant role. Indeed, such bounds are valid for non-regular backward orbits contained in the boundary of the petal. The same applies for the upper bound in Theorem \ref{thm:backward}(c). Other non-regular backward orbits, such as the ones converging to super-repelling fixed points, will be studied in subsequent subsections. As a matter of fact, we will see that the rate of convergence of such backward orbits might be even faster than $e^{\nu t}$, $\nu < 0$.
\end{remark}

\begin{remark}\label{remark on groups, backward}
    A similar remark to Remark \ref{remark on groups, forward} holds for backward orbits as well. As a matter of fact, if $(\phi_t)$ is a non-elliptic group, then it has a unique petal $\Delta$ and $\Delta=\D$. Therefore, backward orbits behave exactly as forward orbits and we get the same rates.
\end{remark}

We move on to the sharpness of the bounds in Theorem \ref{thm:backward}. First notice that the statement itself of Theorem \ref{thm:backward}(b) provides the sharpness for parabolic semigroups of positive hyperbolic step. So, we only need to work with cases (a) and (c). We commence with an example showing the sharpness of the upper bound in Theorem \ref{thm:backward}(a).

\begin{example}
Consider the ``Koebe-like'' domain $\Omega = \C \setminus (-\infty,0]$. Evidently, this is a convex in the positive direction simply connected domain, and so we can relate $\Omega$ to a non-elliptic semigroup $(\phi_t)$ as usual. In particular, since $\Omega$ is not contained inside any horizontal half-plane, $(\phi_t)$ must be parabolic of zero hyperbolic step.
Let $F \colon \H \to \Omega$, where $\H$ is the upper half-plane, be the conformal mapping given by $F(z) = -z^2$. Consider also the M\"obius transform $S \colon \D \to \H$ satisfying $S(0) = i$, $S(1) = \infty$ and $S(-1) = 0$. Then, we can define the semigroup $(\phi_t)$ with Koenigs function $h:=F\circ S$ given by
\begin{equation}\label{eq:koebelike}
\phi_t(z) = h^{-1}(h(z)+t) = \dfrac{\sqrt{\left(\frac{1+z}{1-z}\right)^2+t}-1}{\sqrt{\left(\frac{1+z}{1-z}\right)^2+t}+1}, \quad z \in \D, \, t \geq 0,
\end{equation}
where we are using the branch of the square root that takes $-\Omega$ onto $\H$. Notice that the Denjoy-Wolff point of the semigroup is $1$. By the geometry of $\Omega$, we can see that the semigroup $(\phi_t)$ possesses two parabolic petals, given by the lower and upper parts of the unit disk. For $z$ inside any of the petals, its backward orbit $\widetilde{\gamma}_z$ may be defined via $\widetilde{\gamma}_z(t)=h^{-1}(h(z)-t)$, for all $t\ge0$. For example, suppose that $z \in\Delta := \D \cap \H$. It is possible to see by (\ref{eq:koebelike}) that
$$\sqrt{t}\abs{\widetilde{\gamma}_z(t)-1} = \dfrac{2\sqrt{t}}{\abs{\sqrt{\left(\frac{1+z}{1-z}\right)^2-t}+1}} \to 2, \quad \text{as } t \to +\infty.$$
\end{example}

\medskip

The sharpness of the lower bound in Theorem \ref{thm:backward}(a) can be derived by using similar arguments as in \cite[Theorem 5.6]{BCDM-Rate}, but this time for backward orbits. Indeed, using the angular sectors $\{w\in\C:-\theta<\arg w<\pi\}$, where $\theta\in(0,\pi)$, as Koenigs domains, one may prove that the lower rate can reach arbitrarily close to $1/t$.

\medskip

We end the subsection with a remark with regard to the sharpness of Theorem \ref{thm:backward}(c).

\begin{remark}\label{rmk:conformality petal}
    In an analogue manner to Remark \ref{rmk:conformality Denjoy}, we may proceed to a distinction within the class of hyperbolic petals. Suppose that $\Delta$ is a hyperbolic petal of a non-elliptic semigroup $(\phi_t)$ and let $\sigma$ be the unique repelling fixed point with $\sigma\in\partial\Delta\cap\partial\D$. In \cite{GKR} the authors provide a definition for the conformality of $\Delta$ at $\sigma$. Indeed, the petal $\Delta$ is called \textit{conformal} at $\sigma$ if there exists a conformal mapping $\phi$ of $\D$ onto $\Delta$ such that 
    \begin{equation*}
        \angle\lim\limits_{z\to 1}\phi(z)=\sigma \quad\text{and}\quad \angle\lim\limits_{z\to 1}\frac{\phi(z)-\sigma}{z-1}\in\C.
    \end{equation*}
    This notion of conformality is actually independent of $\phi$ and is an intrinsic geometric property of $\Delta$. According to \cite[Subsection 6.1]{GKR}, if $\Delta$ is conformal at $\sigma$, then $\lim_{t\to+\infty}(e^{-\nu t}|\widetilde{\gamma}_z(t)-\sigma|)$, where $\nu$ is the repelling spectral value at $\sigma$, exists finitely and does not vanish. Hence, the $\epsilon$ in the upper bound of Theorem \ref{thm:backward}(c) can be removed and we obtain a better estimate of the rate. On the other hand, if $\Delta$ is not conformal at $\sigma$, then the $\epsilon$ is necessary and Theorem \ref{thm:backward}(c) provides a sharp upper bound.
\end{remark}

To end the subsection, we notice that Theorem \ref{thm:backward}(a) and (b) can be used to extend \cite[Proposition 4.20]{BCDMG} to parabolic petals.
\begin{corollary}
Let $(\phi_t)$ be a non-elliptic semigroup in $\D$ with a parabolic petal $\Delta$. If $z \in \Delta$, then
$$\lim_{t \to +\infty}\left(\dfrac{1}{t}\log\abs{\widetilde{\gamma}_z(t)-\tau}\right) = 0.$$
\end{corollary}

\subsection{Non-regular Backward Orbits}

We now turn our attention towards non-regular backward orbits. Keeping in mind the geometry of the Koenigs domains of non-elliptic semigroups, as well as the geometry of the image through the Koenigs function of the respective petals, we undestand that we may classify non-regular backward orbits into three categories:
\begin{enumerate}
    \item[(i)] either the backward orbit lies on the boundary of a parabolic petal and converges tangentially to the Denjoy-Wolff point of the semigroup (see \cite[Theorem 13.1.7]{BCDM}),
    \item[(ii)] or the backward orbit lies on the boundary of a hyperbolic petal and converges tangentially to a repelling fixed point of the semigroup (see \cite[Theorem 13.4.12]{BCDM}),
    \item[(iii)] or the backward orbit is not contained in the closure of any petal and converges to a super-repelling fixed point of the semigroup.
\end{enumerate}
Non-regular backward orbits belonging to the first category appear only in parabolic semigroups. The other two types can emerge in any type of semigroup (even in the elliptic case later on).

Suppose that $(\phi_t)$ is a non-elliptic semigroup in $\D$ with Koenigs function $h$ and Koenigs domain $\Omega$. Let $\widetilde{\gamma}_z:[0,+\infty)\to\D$ be a non-regular backward orbit for $(\phi_t)$. Then, it is easy to comprehend that
\begin{equation}\label{eq:non-regular distance}
\lim\limits_{t\to+\infty}\delta_\Omega(h(\widetilde{\gamma}_z(t)))=\lim\limits_{t\to+\infty}\delta_\Omega(h(z)-t)=0,
\end{equation}
regardless of the type of the semigroup or the type of the backward orbit. On the other hand, this trait is not shared by regular backward orbits. Indeed, their inclusion inside a petal forces the above limit to be positive. 

Despite this difference, the rates of convergence for non-regular backward orbits can be derived in similar fashion to the respective rates for regular backward orbits.

\begin{proof}[Proof of Theorem \ref{thm:non-regular}]
     Inspecting closely the proofs of Theorem \ref{thm:backward}(a) and (b), we see that the inclusion of the starting point $z$ inside a parabolic petal $\Delta$ did not factor when executing all the necessary steps to extract the desired result. Indeed, the two proofs work exactly in the same manner whenever $z\in\partial\Delta$. Ergo the desired rates for (a) and (b) can be deduced at once. In addition, a similar remark is true for the proof of Theorem \ref{thm:backward}(c) in case the starting point $z$ lies on the boundary of a hyperbolic petal. As a result, the desired rate in (c) is also true.

     (d) All that remains is to work with $\widetilde{\gamma}_z$ when it converges to a super-repelling fixed point $\sigma\in\partial\D\setminus\{\tau\}$. In this case, the repelling spectral value of $(\phi_t)$ at $\sigma$ can be said to be $-\infty$. Again, we work as in the proof of Theorem \ref{thm:backward}(c). Through the diameter estimate for harmonic measure and assuming without loss of generality that $h(0)=0$, we may write
    $$|\widetilde{\gamma}_z(t)-\sigma|\le2\pi\omega(0,B_t,\Omega\setminus B_t),$$
    for all $t\ge0$, where $B_t=\{h(z)-s:s\in[t,+\infty)\}$. Due to the convexity of $\Omega=h(\D)$ in the positive direction and the fact that $\lim_{t\to+\infty}\delta_\Omega(h(z)-t)=0$ with $z$ not contained in the closure of any petal, there exists some $T\ge0$ such that for each $t\ge T$ we are able to find points $p_t^+, p_t^-\in\partial\Omega$ with $\text{Re}p_t^+=\text{Re}p_t^-=\text{Re}h(z)-t$ and $\text{Im}p_t^-<\text{Im}h(z)<\text{Im}p_t^+$. In case there are multiple points of $\partial\Omega$ satisfying these conditions, we choose the ones the are closest to $h(z)-t$ in Euclidean terms. Set $d_t=\text{Im}p_t^+-\text{Im}p_t^-$. Evidently, $\lim_{t\to+\infty}d_t=0$. Since $p_T^+,p_T^-\in\partial\Omega$, the whole horizontal half-lines $Q_T^+=\{p_T^+-s:s\ge0\}$ and $Q_T^-=\{p_T^--s:s\ge0\}$ are contained in $\C\setminus\Omega$, because of the convexity. So $\Omega\subset\C\setminus(Q_T^+\cup Q_T^-)=:\Omega_T$. Therefore, the domain monotonicity property for harmonic measure yields
    $$\omega(0,B_t,\Omega\setminus B_t)\le\omega(0,B_t,\Omega_T\setminus B_t).$$
    For $t\ge T$, set $D_t=\{\text{Re}h(z)-t+i\text{Im}p_T^-+is:s\in[0,d_T]\}$. Then, by the maximum principle for harmonic functions, we obtain
    $$\omega(0,B_t,\Omega_T\setminus B_t)\le\omega(0,D_t,A_T),$$
    where $A_T$ denotes the simply connected domain bounded by $\partial\Omega_T$ and $D_t$ and that is not fully contained in the horizontal strip defined by the carriers of $Q_T^+,Q_T^-$. Binding everything together, we have found that
    \begin{equation}\label{eq:last1}
    |\widetilde{\gamma}_z(t)-\sigma|\le2\pi\omega(0,D_t,A_T).
    \end{equation}
    Using Theorem \ref{thm:Beurling}, one has
    \begin{equation}\label{eq:last2}
    \omega(0,D_t,A_T)\le\frac{8}{\pi}\exp\left(-\pi\lambda_{A_T\setminus[0,+\infty)}(D_t,[0,+\infty))\right),
    \end{equation}
    since the horizontal half-line $[0,+\infty)$ joins $0$ to $\partial A_T\setminus D_t$.
    However, combining Lemma \ref{lemma:serial-rule} and Example \ref{ex:rectangle}, we get that
    \begin{equation}\label{eq:last3}
    \lambda_{A_T\setminus[0,+\infty)}(D_t,[0,+\infty))\ge\lambda_R(D_t,D_T)=\frac{t-T}{d_T},
    \end{equation}
    where $R$ denotes the rectangle bounded by $D_t,D_T,Q_T^+$ and $Q_T^-$. As a result, by merging \eqref{eq:last1}, \eqref{eq:last2} and \eqref{eq:last3}, we obtain
    $$|\widetilde{\gamma}_z(t)-\sigma|\le 16 e^{-\pi\frac{t-T}{d_T}}=16e^{\frac{\pi T}{d_T}}e^{-\frac{\pi}{d_T}t},$$
    for all $t>T$. Nevertheless, the fact that $\lim_{t\to+\infty}\delta_\Omega(h(z)-t)=0$ dictates that by enlarging $T$, the width $d_T$ can reach arbitrarily close to $0$. As a consequence, letting $\epsilon>0$ and choosing $T$ such that $d_T<\pi/\epsilon$, we get the desired result for $t>T$. Finally, for $t\in[0,T]$, as in the previous proofs, we may write
    $$|\widetilde{\gamma}_z(t)-\sigma|\le 2=2e^{\frac{\pi}{d_T}t}e^{-\frac{\pi}{d_T}t}\le 2e^{\frac{\pi T}{d_T}}e^{-\frac{\pi}{d_T}t}\le2e^{\frac{\pi T}{d_T}}e^{-\epsilon t}.$$
    Thus, setting $C=C(z,\epsilon)=16e^{\pi T/d_T}$ (recall that $T$ depends on $z$ and $\epsilon$), we have the desired result for all $t\ge0$.

\end{proof}

As one may observe, we did not provide any lower bound for the rate of convergence of non-regular backward orbits. This is because non-regular backward orbits of non-elliptic semigroups may converge arbitrarily fast to the fixed point at which they are going to land. We may verify this through the next example.

\begin{example}\label{ex:non-regular}
Consider an increasing smooth map $\theta \colon \R \to (0,+\infty)$ such that
$$\lim_{x \to -\infty}\theta(x) = 0, \qquad \lim_{x \to +\infty}\theta(x) = +\infty.$$
Define $\Omega := \{x+iy \in \C : x \in \R, \, \abs{y} < \theta(x)\}$. By construction, $\Omega$ is a domain which is starlike at infinity and which is symmetric with respect to the real line (i.e., $\Omega = \{\overline{z} : z \in \Omega\}$). Consider a Riemann mapping $h \colon \D \to \Omega$ such that $h((-1,1)) = \R$ and $h'(0) > 0$. Let $(\phi_t)$ be the semigroup in $\D$ given by
$$\phi_t(z) = h^{-1}(h(z)+t), \qquad t \ge 0, \, z \in \D.$$
It is clear that $(\phi_t)$ is a non-elliptic semigroup with Denjoy-Wolff point $\tau = 1 \in \partial\D$. Since $\Omega$ is not contained in any horizontal strip or horizontal half-plane, $(\phi_t)$ must be a parabolic semigroup of zero hyperbolic step. We also notice that $\sigma = -1 \in \partial\D$ is a boundary fixed point of $(\phi_t)$. Indeed, by the geometry of $\Omega$, $\sigma$ is a super-repelling point. In particular, let $z \in (-1,1)$, and consider the backward orbit $\widetilde{\gamma}_z \colon [0,+\infty) \to \D$ starting at $z$. We have that $\lim_{t \to +\infty}\widetilde{\gamma}_z(t) = \sigma$. Notice that $h\circ\widetilde{\gamma}_z([0,+\infty))$ lies on the axis of symmetry of $\Omega$ and hence $\widetilde{\gamma}_z(t)$ converges non-tangentially (in fact, orthogonally) to $\sigma$, as $t \to +\infty$. Then $\widetilde{\gamma}_z([0,+\infty))$ is contained in some Stolz angle of $\D$ with vertex $\sigma$. Therefore, there must exist $K > 0$ so that
\begin{equation}\label{eq:stolz}
\abs{\widetilde{\gamma}_z(t)-\sigma} \leq K(1-\abs{\widetilde{\gamma}_z(t)}), \qquad t \ge 0.
\end{equation}
As we have done previously, using the triangle inequality for the hyperbolic distance and its conformal invariance, we get that
\begin{align}
\abs{\widetilde{\gamma}_z(t)-\sigma} & \leq K(1-\abs{\widetilde{\gamma}_z(t)}) \leq 2K\dfrac{1-\abs{\widetilde{\gamma}_z(t)}}{1+\abs{\widetilde{\gamma}_z(t)}} = 2K\exp(-2d_{\D}(0,\widetilde{\gamma}_z(t))) \nonumber \\
& \leq 2K\exp(2d_{\D}(0,z))\exp(-2d_{\D}(z,\widetilde{\gamma}_z(t))) \nonumber \\
& = 2K\exp(2d_{\D}(0,z))\exp(-2d_{\Omega}(h(z),h(z)-t)), \qquad t \ge 0.
\label{eq:first-estimate-super-repelling}
\end{align}

But, using Theorem \ref{DistanceLemma}, we get that
$$d_{\Omega}(h(z),h(z)-t) \geq \dfrac{1}{4}\log\left(1+\dfrac{t}{\min\{\delta_{\Omega}(h(z)),\delta_{\Omega}(h(z)-t)\}}\right), \quad t \ge 0.$$
By construction, for all $t \ge 0$ we have that
$$\delta_{\Omega}(h(z)-t) \leq \delta_{\Omega}(h(z)) \;\text{ and }\; \delta_{\Omega}(h(z)-t) \leq \theta(\textup{Re}h(z)-t).$$
So,
\begin{equation}\label{eq:last distance lemma}
d_{\Omega}(h(z),h(z)-t) \geq \dfrac{1}{4}\log\left(1+\dfrac{t}{\theta(\textup{Re}h(z)-t)}\right).
\end{equation}
Therefore, using \eqref{eq:first-estimate-super-repelling} and \eqref{eq:last distance lemma}, we get that
$$\abs{\widetilde{\gamma}_z(t)-\sigma} \leq 2K\dfrac{1+\abs{z}}{1-\abs{z}}\sqrt{\dfrac{\theta(\textup{Re}h(z)-t)}{t+\theta(\textup{Re}h(z)-t)}} \leq 2K\dfrac{1+\abs{z}}{1-\abs{z}}\sqrt{\dfrac{\theta(\textup{Re}h(z)-t)}{t}}.$$
But $\theta(\textup{Re}h(z)-t)$ may converge to $0$ arbitrarily fast or slow, as $t\to+\infty$. Therefore, the quantity $|\widetilde{\gamma}_z(t)-\sigma|$ cannot be bounded below by the same rate for all semigroups.
\end{example}

Recall that the rate of convergence of a regular backward orbit in a hyperbolic petal can be quantified using the spectral value $\nu \in (-\infty,0)$ of its associated repelling fixed point; see \cite[Proposition 4.20]{BCDMG}. On the other hand, as we said before, the spectral value of super-repelling fixed points is $\nu = -\infty$. Using this convention, we give the following consequence of Theorem \ref{thm:non-regular}(d), which extends \cite[Proposition 4.20]{BCDMG} to the case of super-repelling fixed points.
\begin{corollary}
\label{cor:super-repelling}
Let $(\phi_t)$ be a non-elliptic semigroup in $\D$ with a super-repelling fixed point $\sigma \in \partial\D$. Suppose that $z \in \D$ is such that $\lim_{t\to+\infty}\widetilde{\gamma}_z(t)=\sigma$. Then
$$\lim_{t \to +\infty}\left(\dfrac{1}{t}\log\abs{\widetilde{\gamma}_z(t)-\sigma}\right) = -\infty.$$
\end{corollary}

\section{Elliptic Semigroups}
We conclude our study on the rates of convergence by dealing with elliptic semigroups. The desired result for forward orbits is mostly known (see \cite{Gurganus} and \cite[Proposition 4.4.2, Remark 4.4.4]{Shoikhet}) and its proof is entirely different than the ones we provided in the non-elliptic case. This difference relies mostly on the fact that this time the Denjoy-Wolff point lies inside the unit disk. As mentioned before, we will still provide the proof of the result for the sake of completeness. On the other hand, the result for the backward orbits is derived via the \textit{lifting technique} that correlates an elliptic semigroup in $\D$ with a non-elliptic semigroup in the right half-plane $\C_0$. This technique is a remarkable result introduced by Gumenyuk in \cite[Proposition 2.1]{Gumenyuk} for semigroups. A similar technique also appeared in \cite{CDM} for Loewner chains. More specifically, we will approximate the quantity $|\widetilde{\gamma}_z(t)-\sigma|$ for an elliptic semigroup $(\phi_t)$ with a regular backward orbit $\widetilde{\gamma}_z$ converging to a repelling fixed point $\sigma\in\partial\D$ by correlating it with the quantity $|\gamma(t)-\hat{\sigma}|$, where $\gamma$ is a regular backward orbit contained in a hyperbolic petal of a non-elliptic semigroup $(\hat{\phi}_t)$ with repelling fixed point $\hat{\sigma}$. The semigroup $(\hat{\phi}_t)$ is going to be defined via the initial semigroup $(\phi_t)$, despite their striking difference in type.

\begin{proof}[Proof of Theorem \ref{thm:elliptic}]

(a) We start by assuming that $\tau = 0$. In that case, using (\ref{infinitesimal generator}) and (\ref{berkson-porta}), the semigroup $(\phi_t)$ can be written as the solution of the differential equation
\begin{equation}\label{eq:diff eq}
\dfrac{\partial\phi_t(z)}{\partial t} = -\phi_t(z)p(\phi_t(z)), \qquad z \in \D,
\end{equation}
where $\phi_0(z) = z$ for all $z \in \D$ and $p$ is a holomorphic function that maps $\D$ into $\C_0$ (recall that $(\phi_t)$ is not a group).

Using a standard argument (see for instance \cite[Proof of Lemma in p.52]{Duren}), it is possible to see that
\begin{equation}\label{eq:standard argument}
\abs{z}e^{-\epsilon_{1,z}t} \leq \abs{\phi_t(z)} \leq \abs{z}e^{-\epsilon_{2,z}t}, \qquad t \geq 0, \, z \in \D,
\end{equation}
where $\epsilon_{1,z} = \max\{\mathrm{Re}p(w) : \abs{w} \leq \abs{z}\}$ and $\epsilon_{2,z} = \min\{\mathrm{Re}p(w) : \abs{w} \leq \abs{z}\}$. 

By \cite[Corollary 10.1.12.(1)]{BCDM}, we note that $p(0) = \lambda$, the spectral value of $(\phi_t)$. Thus, the function
$$p_1(z) := \dfrac{p(z)-i\mathrm{Im}\lambda}{\mathrm{Re}\lambda}$$
is holomorphic, maps $\D$ into $\C_0$ and satisfies $p_1(0) = 1$. In that case, a distortion theorem (see \cite[Theorem 2.2.1]{BCDM}) assures that
\begin{equation}\label{eq:distortion half-plane}
\mathrm{Re}\lambda\dfrac{1-\abs{z}}{1+\abs{z}} \leq \mathrm{Re}p(z) = \mathrm{Re}\lambda\mathrm{Re}p_1(z) \leq \mathrm{Re}\lambda\dfrac{1+\abs{z}}{1-\abs{z}}.
\end{equation}
It follows that
\begin{equation}\label{eq:last elliptic}
\epsilon_{1,z} \leq \mathrm{Re}\lambda\dfrac{1+\abs{z}}{1-\abs{z}}, \qquad \epsilon_{2,z} \geq \mathrm{Re}\lambda\dfrac{1-\abs{z}}{1+\abs{z}}.
\end{equation}
Hence, combining \eqref{eq:standard argument} with \eqref{eq:last elliptic},
\begin{align*}
\abs{\gamma_z(t)} & = \abs{\phi_t(z)} \leq \abs{z}\exp\left(-\mathrm{Re}\lambda\dfrac{1-\abs{z}}{1+\abs{z}}t\right) = \abs{z}\exp\left(-\mathrm{Re}\lambda\exp(-2d_{\D}(0,z))t\right),
\end{align*}
and
\begin{align*}
\abs{\gamma_z(t)} & = \abs{\phi_t(z)} \geq \abs{z}\exp\left(-\mathrm{Re}\lambda\dfrac{1+\abs{z}}{1-\abs{z}}t\right) = \abs{z}\exp\left(-\mathrm{Re}\lambda\exp(2d_{\D}(0,z))t\right).
\end{align*}

Let us now address the general case in which $\tau \in \D$. To do so, define $\Psi \colon \D \to \D$  by
\begin{equation}\label{eq:automorphism}
\Psi(z) = \dfrac{\tau-z}{1-\overline{\tau}z}, \quad z \in \D.
\end{equation}
Notice that $\Psi = \Psi^{-1}$ and $\Psi(\tau) = 0$. Consider the semigroup $\widetilde{\phi}_t = \Psi \circ \phi_t \circ \Psi$, and notice that it is also elliptic, with Denjoy-Wolff point $0$ and spectral value $\lambda$. Therefore, by the previous arguments, 
$$\abs{z}\exp\left(-\mathrm{Re}\lambda\exp(2d_{\D}(0,z))t\right) \leq |\widetilde{\phi}_t(z)| \leq \abs{z}\exp\left(-\mathrm{Re}\lambda\exp(-2d_{\D}(0,z))t\right).$$
Then, due to \eqref{eq:automorphism},
\begin{align*}
\abs{\dfrac{\tau-\phi_t(z)}{1-\overline{\tau}\phi_t(z)}} = \abs{\Psi \circ \phi_t(z)} & \leq \abs{\Psi(z)}\exp\left(-\mathrm{Re}\lambda\exp(-2d_{\D}(0,\Psi(z)))t\right) \\
& = \abs{\Psi(z)}\exp\left(-\mathrm{Re}\lambda\exp(-2d_{\D}(\tau,z))t\right).  
\end{align*}
Similarly,
\begin{align*}
\abs{\dfrac{\tau-\phi_t(z)}{1-\overline{\tau}\phi_t(z)}} = \abs{\Psi \circ \phi_t(z)} & \geq \abs{\Psi(z)}\exp\left(-\mathrm{Re}\lambda\exp(2d_{\D}(0,\Psi(z)))t\right) \\
& = \abs{\Psi(z)}\exp\left(-\mathrm{Re}\lambda\exp(2d_{\D}(\tau,z))t\right).  
\end{align*}
Since
$$1-\abs{\tau} \leq 1-\abs{\tau}\abs{\phi_t(z)} \leq \abs{1-\overline{\tau}\phi_t(z)} \leq 1+\abs{\tau}\abs{\phi_t(z)} \leq 1+\abs{\tau},$$
the desired result follows.

(b) Next, suppose that $z$ is contained inside the hyperbolic petal $\Delta$ of $(\phi_t)$. By potentially using two automorphisms of the unit disk, we may assume that $\tau=0$ and $\sigma=1$. Then, by \cite[Proposition 2.1]{Gumenyuk},  there exists a non-elliptic semigroup $(\psi_t)$ on the right half-plane $\C_0$ with Denjoy-Wolff point $\infty$ such that
\begin{equation}\label{eq:lifting}
e^{-\psi_t(w)}=\phi_t(e^{-w}), \quad\text{for all }w\in\C_0.
\end{equation}
Set $F(w)=e^{-w}$. The mapping $F$ maps the half-strip $\Sigma=\{w\in\C:|\text{Im}w|<\pi, \; \text{Re}w>0\}$ conformally onto the slit unit disk $\D\setminus(-1,0]=:\D_+$. Even though $F$ can be defined for the whole complex plane, from now on, we only consider its restriction on $\Sigma$. Since $\lim_{t\to+\infty}\widetilde{\gamma}_z(t)=\sigma=1$, there exists a $T\ge0$ such that $|\log\widetilde{\gamma}_z(t)|<1$, for all $t\ge T$. By extension, $\widetilde{\gamma}_z(t)\in\D_+$, for all $t\ge T$. Obviously $F^{-1}(1)=0$ and as a result, there exists a curve $\delta:[T,+\infty)\to\Sigma$ such that $F(\delta(t))=\widetilde{\gamma}_z(t)$ and $\lim_{t\to+\infty}\delta(t)=0$. So $e^{-\delta(t)}=\widetilde{\gamma}_z(t)$ and we are left with evaluating the quantity
\begin{equation}\label{eq:delta}
|e^{-\delta(t)}-1|, \quad t\ge T.
\end{equation}
Let $t> s> T$ with $t-s>T$. The formal definition of backward orbits yields 
$\widetilde{\gamma}_z(t-s)=\phi_s(\widetilde{\gamma}_z(t))$. Then
\begin{eqnarray*}
\delta(t-s)&=&F^{-1}(\widetilde{\gamma}_z(t-s))=F^{-1}(\phi_s(\widetilde{\gamma}_z(t)))=F^{-1}(\phi_s(e^{-\delta(t)}))\\ &=&F^{-1}(e^{-\psi_s(\delta(t))})=F^{-1}(F(\psi_s(\delta(t))))=\psi_s(\delta(t)).
\end{eqnarray*}
 As a consequence, $\delta$ is a backward orbit for the non-elliptic semigroup $(\psi_t)$. Our next goal is to move away from $\C_0$ and return to $\D$. For this reason, consider the Cayley transform $C:\D\to\C_0$ with $C(z)=\frac{1+z}{1-z}$. Through the conjugation
$$\Hat{\phi}_t=C^{-1}\circ\psi_t\circ C,$$
we construct a semigroup $(\hat{\phi}_t)$ in $\D$. It is easy to check that the new semigroup is also non-elliptic with Denjoy-Wolff point $C^{-1}(\infty)=1$ and repelling fixed point $C^{-1}(0)=-1$. Plus, following similar arguments as before, the curve $\gamma:[T,+\infty)\to\D$ with $\gamma(t)=C^{-1}(\delta(t))$ can be quickly verified to be a backward orbit for $(\hat{\phi}_t)$. At once
\begin{equation}\label{eq:delta2}
\delta(t)=C(\gamma(t))=\frac{1+\gamma(t)}{1-\gamma(t)}, \quad\text{for all }t\ge T,
\end{equation}
which signifies that because of \eqref{eq:delta} and \eqref{eq:delta2} we need to find an upper bound for
$$\left|e^{-\frac{1+\gamma(t)}{1-\gamma(t)}}-1\right|.$$
However, for $t\ge T$, we have $|\log\widetilde{\gamma}_z(t)|<1$, which in turn leads to $|\delta(t)|<1$. But for all $z\in\D$, it is easy to check that $(3-e)|z|\le |e^z-1|\le(e-1)|z|$, which implies that
$$\frac{3-e}{2}|1+\gamma(t)|\le(3-e)\frac{|1+\gamma(t)|}{|1-\gamma(t)|}\le\left|e^{-\frac{1+\gamma(t)}{1-\gamma(t)}}-1\right|\le\frac{e-1}{1-\text{Re}\gamma(t)}|1+\gamma(t)|.$$
Through the relation between $\gamma$ and $\delta$ it is easy to compute that
\begin{eqnarray*}
    1-\text{Re}\gamma(t)=\frac{2\text{Re}\delta(t)+2}{(\text{Re}\delta(t)+1)^2+(\text{Im}\delta(t))^2}&\ge& \frac{2}{(\text{Re}\delta(T)+1)^2+\pi^2}\\
    &=&\frac{2}{(\log|\Tilde{\gamma}_z(T)|+1)^2+\pi^2}.\end{eqnarray*}
Combining everything together, we have that
$$\frac{3-e}{2}|\gamma(t)+1|\le|\widetilde{\gamma}_z(t)-1|\le\frac{\left((\log|\Tilde{\gamma}_z(T)|+1)^2+\pi^2\right)(e-1)}{2}|\gamma(t)+1|,$$
for all $t\ge T$. Finally, by the conjugation formulas of $F$ and $C$, it is simple to check that $\angle\lim_{z\to -1}\hat{\phi}_t^{'}(z)=\angle\lim_{z\to1}\phi_t^{'}(z)=e^{-\nu t}$, since $\nu\in(-\infty,0)$ is the repelling spectral value of $(\phi_t)$ at $\sigma=1$. So $-1$ is a repelling fixed point for $(\hat{\phi}_t)$ with repelling spectral value $\nu$, as well. By what we have already proved in Theorem \ref{thm:backward}(c), we have that for every $\epsilon\in(0,-\nu)$, there exists some constants $C_1=C_1(z)$ and $C_2=C_2(z,\epsilon)>0$ which we have explicitly stated such that
$$C_1e^{\nu t}\le|\gamma(t)+1|\le C_2e^{(\nu+\epsilon)t}.$$
With this in mind, we deduce the desired result, for $t>T$, for the initial elliptic semigroup $(\phi_t)$ with
\begin{equation*}
    c_1=\frac{3-e}{2}\frac{1-|\zeta|}{1+|\zeta|}\sin^2[\nu(\textup{Im}\hat{h}(\zeta)-\alpha)], \; c_2=8(e-1)((\log|\Tilde{\gamma}_z(T)|+1)^2+\pi^2)e^{-(\nu+\epsilon)T},
\end{equation*}
where $\hat{h}$ is the Koenigs function of $(\hat{\phi}_t)$, $\zeta=\gamma(T)=\frac{\delta(T)-1}{\delta(T)+1}=\frac{\log \Tilde{\gamma}_z(T)+1}{\log \Tilde{\gamma}_z(T)-1}$ and the roles of the numbers $\alpha$ and $T$ have been explained. Following in the footsteps of the previous proofs, the result can be trivially expanded for all $t\ge0$.
\end{proof}

\begin{remark}\label{rmk: elliptic petal}
    A similar remark to Remark \ref{rmk:conformality petal} about hyperbolic petals and conformality is valid in elliptic semigroups, as well.
\end{remark}

\begin{remark}
    The proof of Theorem \ref{thm:elliptic}(b) with the use of the lifting technique may be used in an identical manner in order to derive upper bounds for the rate of non-regular backward orbits of elliptic semigroups. Since the proof is basically the same, we omit it for the sake of not being redundant.
\end{remark}

Finally, we need to prove that the results of Theorem \ref{thm:elliptic} are sharp. For the backward orbits, the sharpness is derived from the corresponding sharpness in Theorems \ref{thm:backward} and \ref{thm:non-regular} (for the regular backward orbits and the non-regular backward orbits, respectively). 

On the other hand, notice that Theorem \ref{thm:elliptic}(a) is the only case where the rate depends on the starting point $z$. In any other case, $z$ influenced solely the constants which govern the convergence to the respective fixed point. The sharpness of the upper bound can be easily explored. Indeed, consider an elliptic semigroup $(\phi_t)$ with spectral value $\lambda$, Denjoy-Wolff point $0$ and repelling fixed point $\sigma\in\partial\D$. It is easy to check that
\begin{equation}\label{eq:elliptic upper sharpness}
\lim\limits_{z\to\sigma}\frac{|\gamma_z(t)|}{|z|\exp\left(-\textup{Re}\lambda\frac{1-|z|}{1+|z|}t\right)}=1,
\end{equation}
for all $t\ge0$, which proves the desired sharpness and the necessity of the factor $\frac{1-|z|}{1+|z|}=\exp(-2d_\D(0,z))$ in the bound. Hence, we understand, that close to repelling fixed points, elliptic semigroups attain their slowest rate.

All that remains is to inspect the sharpness for the lower bound. We will achieve this through the next and final example.

\begin{example}
Fix $\lambda>0$ and consider the elliptic semigroup $(\phi_t)$ given by the formula
\begin{equation}
\label{formula}
\phi_t(z)=\frac{4e^{-\lambda t}z}{2(1-z)\sqrt{(1-z)^2+4e^{-\lambda t}z}+2(1-z)^2+4e^{-\lambda t}z}, \quad z\in\D,\; t\ge0.
\end{equation}
It can be readily checked through straightforward calculations that $(\phi_t)$ is indeed an elliptic semigroup with Denjoy-Wolff point $0$, spectral value $\lambda$, Koenigs function $h(z)=\frac{4z}{(1-z)^2}$, and Koenigs domain $\Omega:=\C\setminus(-\infty,-1]$.

Choose $z \in (-1,0)$, and notice by \eqref{formula} that
$$\abs{\dfrac{\phi_t(z)}{z}} = \frac{4e^{-\lambda t}}{2(1-z)\sqrt{(1-z)^2+4e^{-\lambda t}z}+2(1-z)^2+4e^{-\lambda t}z}, \qquad t\ge0.$$
Then, it is possible to check that
$$\lim_{t \to 0^+}\dfrac{\log\abs{\gamma_z(t)/z}}{t} = -\lambda\dfrac{1-z}{1+z} = -\lambda\dfrac{1+\abs{z}}{1-\abs{z}}.$$
Therefore, we see that the presence of the hyperbolic distance in the argument of the exponential rate of the lower estimate in Theorem \ref{thm:elliptic}(a) is crucial.
\end{example}

Last but not least, Theorem \ref{thm:elliptic}(a) helps us extract the following brief corollary in the spirit of \cite[Proposition 16.2.1]{BCDM}, which is its counterpart for non-elliptic semigroups.
\begin{corollary}
    Let $(\phi_t)$ be an elliptic semigroup in $\D$ with Denjoy-Wolff point $\tau\in\D$ and spectral value $\lambda\in\C$, $\textup{Re}\lambda>0$. Then, for all $z\in\D\setminus\{\tau\}$,
    $$\lim\limits_{t\to+\infty}\left(\frac{1}{t}\log|\gamma_z(t)-\tau|\right)=-\textup{Re}\lambda.$$
\end{corollary}
\begin{proof}
    For the sake of simplicity in the computations assume that $\tau=0$. If this is not true, the proof can be executed similarly, albeit with some minor modifications. Fix $z\in\D\setminus\{0\}$ and let $t>0$ and $\epsilon>1$. We may write $t=\frac{t}{\epsilon}+\frac{(\epsilon-1)t}{\epsilon}$. Then,
    \begin{equation}\label{eq:cor0}
    \gamma_z(t)=\gamma_z\left(\frac{t}{\epsilon}+\frac{(\epsilon-1)t}{\epsilon}\right)=\gamma_{\gamma_z\left(\frac{(\epsilon-1)t}{\epsilon}\right)}\left(\frac{t}{\epsilon}\right).
    \end{equation}
    Therefore, applying the upper bound of Theorem \ref{thm:elliptic}(a) in \eqref{eq:cor0} yields
    \begin{equation}\label{eq:cor1}
        |\gamma_z(t)|\le\left|\gamma_z\left(\frac{(\epsilon-1)t}{\epsilon}\right)\right|\exp\left(-\textup{Re}\lambda\frac{1-\left|\gamma_z\left(\frac{(\epsilon-1)t}{\epsilon}\right)\right|}{1+\left|\gamma_z\left(\frac{(\epsilon-1)t}{\epsilon}\right)\right|}\frac{t}{\epsilon}\right).
    \end{equation}
    Through a second application of the same theorem, we obtain
    \begin{equation}\label{eq:cor2}
        \left|\gamma_z\left(\frac{(\epsilon-1)t}{\epsilon}\right)\right|\le|z|\exp\left(-\textup{Re}\lambda\frac{1-|z|}{1+|z|}\frac{(\epsilon-1)t}{\epsilon}\right).
    \end{equation}
    Combining relations (\ref{eq:cor1}) and (\ref{eq:cor2}), we find
    $$|\gamma_z(t)|\le|z|\exp\left(-\textup{Re}\lambda\frac{1-|z|}{1+|z|}\frac{(\epsilon-1)t}{\epsilon}\right)\exp\left(-\textup{Re}\lambda\frac{1-\left|\gamma_z\left(\frac{(\epsilon-1)t}{\epsilon}\right)\right|}{1+\left|\gamma_z\left(\frac{(\epsilon-1)t}{\epsilon}\right)\right|}\frac{t}{\epsilon}\right),$$
    for all $z\in\D\setminus\{0\}$, $t>0$ and $\epsilon>1$. As a consequence,
    $$\log|\gamma_z(t)|\le\log|z|-\textup{Re}\lambda\frac{1-|z|}{1+|z|}\frac{(\epsilon-1)t}{\epsilon}-\textup{Re}\lambda\frac{1-\left|\gamma_z\left(\frac{(\epsilon-1)t}{\epsilon}\right)\right|}{1+\left|\gamma_z\left(\frac{(\epsilon-1)t}{\epsilon}\right)\right|}\frac{t}{\epsilon}.$$
    Noticing that $\lim_{t\to+\infty}\gamma_z\left(\frac{(\epsilon-1)t}{\epsilon}\right)=0$, since $\epsilon>1$, we see that
    \begin{equation}\label{eq:cor3}
    \limsup\limits_{t\to+\infty}\left(\frac{1}{t}\log|\gamma_z(t)|\right)\le-\textup{Re}\lambda\frac{1-|z|}{1+|z|}\frac{\epsilon-1}{\epsilon}-\textup{Re}\lambda\frac{1}{\epsilon},
    \end{equation}
    for all $\epsilon>1$. Letting $\epsilon\to1$ in \eqref{eq:cor3}, we deduce $\limsup_{t\to+\infty}(\log|\gamma_z(t)|/t)\le-\textup{Re}\lambda$. Following analogue ideas but this time using the lower bound of Theorem \ref{thm:elliptic}(a), we may also deduce a kind of reverse inequality for the lower limit. Hence, the limit exists and we have the desired limit.
\end{proof}

The backward analogue of the last result can be found in \cite[Proposition 4.20]{BCDMG} for regular backward orbits. For the non-regular ones that are associated to super-repelling fixed points, a similar result can be derived (as in Corollary \ref{cor:super-repelling}) using the arguments in the proof of Theorem \ref{thm:elliptic}(b).

\medskip

\section*{Acknowledgements}

We thank M. D. Contreras and L. Rodr\'{i}guez-Piazza for the insightful remarks they shared with us during the preparation of this work.

\end{document}